\documentclass{amsart}
\usepackage{graphicx} 

\usepackage{amsthm}
\usepackage{amsmath}
\usepackage{amsfonts}
\usepackage{amssymb}
\usepackage{mathrsfs}
\usepackage{mathtools}
\usepackage{tikz-cd}
\usepackage{enumitem}
\usepackage[utf8]{inputenc}
\usepackage[left=30mm,right=30mm]{geometry}
\usepackage{hyperref}
\usepackage{caption}
\usepackage{subcaption}
\usepackage{algorithm}
\usepackage[noend]{algpseudocode}
\usepackage{thm-restate}
\usepackage{mathabx}

\theoremstyle{plain}
\newtheorem{thm}{Theorem}[section]
\newtheorem{lem}[thm]{Lemma}
\newtheorem{prop}[thm]{Proposition}
\newtheorem{cor}[thm]{Corollary}

\theoremstyle{definition}
\newtheorem{defn}[thm]{Definition}
\newtheorem{rem}[thm]{Remark}

\newtheorem{ex}[thm]{Example}

\newcommand{\mc}{\mathcal}
\newcommand{\mf}{\mathfrak}

\newcommand{\Z}{\mathbb{Z}}

\usepackage{xcolor}

\DeclarePairedDelimiter\abs{\lvert}{\rvert}
\DeclarePairedDelimiter\pabs{\lvert}{\rvert_p}
\DeclarePairedDelimiter\norm{\lVert}{\rVert}

\newcommand{\word}[1]{\ensuremath{#1.\mathrm{word}}}
\newcommand{\edge}[1]{\ensuremath{#1.\mathrm{edge}}}

\newcommand{\leng}[1]{\ensuremath{#1.\mathrm{leng}}}

\newcommand{\step}{AlgStep }
\newcommand{\sstep}{InnerStep }

\newcommand{\gar}[1]{{\color{gray} \textbf{Guarantee: #1}}}

\newcommand{\tim}[1]{{\color{orange}  #1}}

\title{Nearly-linear solution to the word problem for 3-manifold groups}

\author{Alessandro Sisto}
	\address{Maxwell Institute and Department of Mathematics, Heriot-Watt University, Edinburgh, UK}
	\email{a.sisto@hw.ac.uk}

\author{Stefanie Zbinden}
   \address{Maxwell Institute and Department of Mathematics, Heriot-Watt University, Edinburgh, UK}
	\email{sz2020@hw.ac.uk}

\begin{document}

\begin{abstract}
    We show that the word problem for any 3-manifold group is solvable in time $O(n\log^3 n)$. Our main contribution is the proof that the word problem for admissible graphs of groups, in the sense of Croke and Kleiner, is solvable in $O(n\log n)$; this covers fundamental groups of non-geometric graph manifolds. Similar methods also give that the word problem for free products can be solved ``almost as quickly'' as the word problem in the factors.
\end{abstract}

\maketitle

\section{Introduction}

The goal of this paper is to complete the proof of the following theorem:

\begin{thm}
\label{thm:3_mflds}
    Let $M$ be a closed connected oriented 3-manifold. Then the word problem in $\pi_1(M)$ can be solved in $O(n\log^3(n))$ on a (multi-tape) Turing machine.
\end{thm}

Note that for closed connected oriented 1- and 2-manifolds the word problem can be solved in linear time (because their fundamental groups are abelian or hyperbolic), while for any $n\geq 4$ there are closed connected oriented $n$-manifolds whose fundamental group has unsolvable word problem. This is because of the well-known fact that any finitely presented group can be realised as the fundamental group of such a manifold, combined with the fact that there are finitely presented groups with unsolvable word problem \cite{Novikov:WP, Boone:WP}. For 3-manifolds, the word problem is known to be solvable, for example as a consequence of the Dehn function being at most exponential (because of the geometrisation theorem), see also \cite{auto-stackable}. For a large class of 3-manifold groups, \cite[Theorem 12.4.7]{word-processing} implies the existence of a quadratic time solution to the word problem, but we are not aware of faster-than-quadratic-time algorithms for arbitrary 3-manifold groups. We note that the word problem for mapping class groups was known to be solvable in quadratic time at the time when the first version of this paper was completed, and is now also known to be solvable in time $O(n\log^3(n))$, again on a Turing machine \cite{wp_mcg}.

In order to prove Theorem \ref{thm:3_mflds} we will use the geometrisation theorem. In particular, there are various cases for $\pi_1(M)$, with several being covered by known results in the literature. For example, in some cases $\pi_1(M)$ is hyperbolic, and it is well-known that the word problem can be solved in linear time (see for instance \cite[Theorem 2.18]{notes:hyp}, and \cite{Holt:real-time} for an even stronger result). In some other cases $\pi_1(M)$ is virtually special, that is, it virtually embeds in a right-angled Artin group, and those also have word problem that can be solved in linear time, see \cite{wp_monoids,wp_raags,conj_raags}. The main case that is not covered is that of closed graph manifolds (whose fundamental groups are not hyperbolic, and in some cases also not virtually CAT(0) \cite{Leeb:not_CAT0} and in particular not virtually special). Fundamental groups of graph manifolds have a certain graph of groups decomposition, and what we will show  is that the word problem can be solved quickly for any fundamental group of an \emph{admissible} graph of groups in the sense of Croke and Kleiner \cite{CK:admissible}, see Definition \ref{def:admissible}. This is a class of graphs of groups containing the ones relevant to Theorem \ref{thm:3_mflds}. Here is the precise statement of our theorem.

\begin{thm}
\label{thm:admissible}
    Let $\mc{G}$ be an admissible graph of groups. Then the word problem in $\pi_1(\mc{G})$ can be solved on a Turing machine in $O(n\log(n))$.
\end{thm}

Another ingredient is given in Section \ref{sec:more_alg}, where, adapting Dehn's algorithm, we show that the word problem for central extensions of hyperbolic groups is solvable in linear time on a Turing machine (Proposition \ref{prop:central_word_problem}). We could not find a reference for this, and we believe this is of independent interest.

We emphasise that we made the choice of using the most classical model of computation, Turing machines, throughout this paper. Other models of computations would require less care in the storage strategies and run-time analysis, but we would not get any improvement for the run-time in Theorem~\ref{thm:admissible}. In this sense, one should read Theorem \ref{def:admissible} as saying that the word problem for $\pi_1(\mc{G})$ can be solved in $O(n\log(n))$ \emph{even on a Turing machine}.

There is one additional ingredient that we need, namely that given a free product $G$ of groups whose word problem can be solved in linear time, the word problem for $G$ is solvable in $O(n\log n)$. This would be a special case of the algorithm for relatively hyperbolic groups given in \cite{Farb:rel_hyp}, but we were not able to reconstruct how to implement the algorithm given there in the claimed run-time. We believe it can be implemented in time at most $nf(n)$ (with the convention of Remark~\ref{rem:change_gen_set}) provided that the word problem in the peripheral subgroups can be solved in time at most $f(n)$ for some convex function $f(n)\geq n$. We explain in detail the difficulties encountered in the implementation in Section~\ref{sec:free_prod}, which already arise in the case of free products. Using tools that we developed to prove Theorem \ref{thm:admissible} we were still able to prove the following:

\begin{restatable}{prop}{freeprod}
\label{prop:free_product}
       Let $\mc G$ be a graph of groups with trivial edge groups, and suppose that the word problem in each vertex group is solvable on a Turing machine in time $\preceq f(n)$ for some convex function $f(n)\geq n$. Then the word problem in $G = \pi_1(\mc G)$ is solvable on a Turing machine in time $\preceq \log (n) f(n)$.
\end{restatable}

\subsection{Outline}

In Section \ref{sec:prelim} we discuss preliminaries on Turing machines and admissible graphs of groups. 

In Section \ref{sec:more_alg} we give additional algorithms, in particular solving the word problem for central extensions, and prove Theorem \ref{thm:3_mflds} using those and Theorem \ref{thm:admissible}.

In Section \ref{sec:free_prod} we explain the difficulties in solving the word problem for free products, which also arise for fundamental groups of admissible graphs of groups. This will explain why in later sections we have to be very careful with reduction procedures, and exactly what problems we should avoid.

In Section \ref{sec:admissible} we prove Theorem \ref{thm:admissible}, and using a very similar method also Proposition \ref{prop:free_product}. The main tool for this are middle derivations, which consist in a controlled sequence of ``reductions'' of a word. The key part of middle derivations is that they provide a careful plan for which reductions to perform at the various steps rather than performing reductions whenever possible. We believe that this tool can be useful to solve the word problem quickly in other classes of graphs of groups (where, in particular, the membership problem for the edge groups in the vertex groups can be solved quickly). An alternative approach was suggested to us by Saul Schleimer, see Remark \ref{rem:divide-and-conquer}.

\subsection*{Acknowledgments} We would like to thank Benson Farb and Saul Schleimer for helpful discussions and supportive comments. We would also like to thank the referees for their useful suggestions. The second author was supported by the European Union (ERC, SATURN, 101076148) and by the Deutsche Forschungsgemeinschaft (DFG, German Research Foundation) under Germany's Excellence Strategy - EXC-2047/1 - 390685813.

\section{Preliminaries}
\label{sec:prelim}

\subsection{General setup and description of Turing Machines}

In this paper, all Turing machines are assumed to be multi-tape Turing machines.

Given a Turing machine $\mc M$ over an alphabet $\mc A$ and with bands $\mc B$, every cell of a band either contains a letter of $\mc A$ or is empty.

\begin{defn}
    Let $\mc M$ be a Turing machine over the alphabet $\mc A$ and let $b\in\mc B$ be one of its bands. We say that $b$ \emph{reads} the word $w = a_0a_1\ldots a_k$, if $a_i\in \mc A$ is the letter in the $i$-th cell of $b$ for all $0\leq i \leq k$ and all other cells of $b$ are empty.
\end{defn}

Assume that $\mc A$ contains $\ast$ and $\#$. If a band $b$ reads $w = \ast w_1 \# w_2 \# \ldots \# w_k \ast $ for words $w_i$ over $\mc A -\{\ast, \#\}$, then we call $w_i$ the $i$-th \emph{segment} of $w$. We use the same terminology if $w$ ends with a $\#$ instead of a $\ast$ or if the last $\ast$ is not there. We say that the \emph{end} of any band is the cell with the highest index which is non-empty. If all cells of a band are empty, its end is cell 0.

In a few places in this paper we will need to store integers on a Turing machine. We will always do this by storing a string of either ``1'' or ``-1''.

With this, we can determine in constant time whether a given integer is negative, 0 or positive, and also adding or subtracting 1 to an integer takes constant time (if the head is at the end of the string). Note that storing binary expansions of integers does not have this last property.

In some places in Section~\ref{sec:admissible}, we describe procedures as a list of steps. In some steps the value of some of the variables, for example $x$ and $y$, changes. For example, we could have a step that sets $ y = 2x$ and then increases $x$ by one. Right after we describe the step, we want to point out some properties of the new (and old) values of the variables. To do so, we denote the old value of the variable by $old(\cdot)$ and the new value by $new(\cdot)$. So in the example above we might want to say that $new(y) = 2old(x)$ or that $new(y) = 2new(x) - 2$.

\begin{rem}
\label{rem:change_gen_set}
We denote by $\preceq$ the usual partial order on monotone functions $f:\mathbb N\to \mathbb R$ where $f\preceq g$ if and only if there exists a positive integer $C$ such that $f(n)\leq C g(Cn+C)+Cn+C$, with corresponding equivalence denoted by $\asymp$. Since rewriting a word in a generating set for a group in terms of a different generating set takes linear time in the length of the word, the time complexity of the word problem for a group does not depend on the generating set up to $\asymp$. Note that $f\preceq n \log^c(n)$ is equivalent to $f\in O(n\log^c(n))$, for $c\geq 1$. With a slight abuse, we will often write ``in time at most $f(n)$'' to mean that the algorithm that we are discussing has run time $g(n)$ for a function $g$ with $g\preceq f$.
\end{rem}

\subsection{Graphs of groups}

We now recall the definition of graph of groups, and set some notation and conventions. A graph in this context is given by a set of vertices and a set of directed edges, so each edge $\alpha$ has an initial vertex $\alpha^-$ and a terminal vertex $\alpha^+$. Also, for each edge $\alpha$ we include the edge $\bar\alpha$ with the reverse orientation, that is, $\bar\alpha^+=\alpha^-$ and $\bar\alpha^-=\alpha^+$. Loops are allowed in the graph, that is there could be edges $\alpha$ with $\alpha^-=\alpha+$, and double edges are also allowed; these are pairs of edges with the same initial and terminal vertices. We will always assume that graphs in our graphs of groups are finite, that is, that they have finitely many vertices and edges, and connected. The vertex set of $\Gamma$ is denoted $V(\Gamma)$, and the edge set is denoted $E(\Gamma)$.

A graph of groups is a tuple $\mathcal G=(\Gamma,\{G_\mu\},\{G_\alpha\},\{\tau_\alpha\})$ where:
\begin{itemize}
    \item $\Gamma$ is a graph.
    \item For each vertex $\mu$ of $\Gamma$ we have a group $G_\mu$ (these are called vertex groups),
    \item Similarly, for each edge $\alpha$ of $\Gamma$ we have a group $G_\alpha$ (these are called edge groups), and we require $G_{\alpha}=G_{\bar\alpha}$.
    \item For each edge $\alpha$ we have an injective homomorphism $\tau_\alpha:G_\alpha\to G_{\alpha^+}$ from the edge group to the vertex group at the terminal vertex of the edge.
\end{itemize}

The fundamental group $\pi_1(\mathcal G)$ of a graph of groups $\mathcal G$ as above is roughly the quotient of the free product of the vertex groups and the fundamental group of the graph by the relations coming from the edge groups. More precisely, fix a spanning tree for $\Gamma$, that is, a tree containing all vertices of $\Gamma$ (this choice will not matter up to isomorphism). Then the fundamental group of $\pi_1(\mathcal G)$ is defined as the quotient of
$$\left(\bigast_{\mu\in V(\Gamma)} G_\mu\right) \ast \left(\bigast_{\alpha\in E(\Gamma)}\langle t_\alpha\rangle\right)$$
by the relations
\begin{enumerate}
    \item $t_\alpha=t_{\bar\alpha}^{-1}$,
    \item $t_\alpha=1$ if $\alpha$ is an edge of the fixed spanning tree, and
    \item $t_\alpha\tau_\alpha(x)t_{\alpha}^{-1}=\tau_{\bar\alpha}(x)$ for all edges $\alpha$ and elements $x\in G_\alpha=G_{\bar\alpha}$.
\end{enumerate}

\subsection{Admissible graphs of groups}

We now recall the definition of admissible graphs of groups from \cite{CK:admissible}. Roughly, these are graphs of groups where the vertex groups are central extensions of hyperbolic groups, and they are glued together along $\mathbb Z^2$ subgroups that are virtually generated by pairs of central directions.

\begin{defn}
\label{def:admissible}
Let $\mc{G} = (\Gamma, \{G_\mu\}, \{G_\alpha\}, \{\tau_{\alpha}\})$ be a graph of groups.  We say $\mc{G}$ is \emph{admissible} if the following hold:
\begin{enumerate}
    \item \label{item:edges_exist} $\Gamma$ contains at least one edge.
	\item\label{item:virt_free} Each vertex group $G_\mu$ is a central extension 
	\[ Z_\mu \hookrightarrow G_\mu \stackrel{\pi_\mu}{\twoheadrightarrow} F_\mu \]
	where $Z_\mu$ is an infinite cyclic group and $F_\mu$ is a non-elementary hyperbolic group. 
	\item\label{item:edge_groups} Each edge group $G_\alpha$ is isomorphic to $\mathbb Z^2$ and has the following property: $\tau_{\alpha}(G_\alpha) = \pi^{-1}_{\alpha^+}(\langle a_{\alpha}\rangle)$ for some infinite cyclic subgroup $\langle a_{\alpha}\rangle < F_{\alpha^+}$. 
	\item\label{item:different_R} $\tau^{-1}_{\alpha}(Z_{\alpha^+}) \cap \tau^{-1}_{\bar{\alpha}}(Z_{\alpha^-})$ is the trivial subgroup of $G_\alpha$.
	\item\label{item:malnormal} For each vertex $\mu$, the collection of subgroups $\{\langle a_{\alpha}\rangle : \alpha^+ = \mu\}$ is malnormal in $F_\mu$ (see below for the definition). 
\end{enumerate}
\end{defn}

Recall that a collection of of subgroups $\{H_i\}$ of a group $G$ is malnormal if the following holds. If $g\in G$ and $i,j$ are such that $gH_ig^{-1}\cap H_j$ is infinite, then $i=j$ and $g\in H_i$.

We collect the main properties we need in the following lemma. For a word $w$, we denote its length by $\abs{w}$.

\begin{restatable}{lem}{edgegroups}
\label{lem:properties_of_vertex_groups}
    Let $\mc G$ be an admissible graph of groups with a fixed generating set for $\pi_1(\mathcal G)$. There exists a constant $K$ such that the following holds. For every edge $\alpha\in E(\Gamma)$, there exist elements $c_{\alpha}, o_{\alpha}\in \tau_{\alpha}(G_{\alpha})$ such that for all edges $\beta \in E(\Gamma)$, the following properties hold.
    \begin{enumerate}
        \item $c_{\alpha}$ and $o_{\alpha}$ generate a copy of $\mathbb Z^2$. Moreover, $\pi_{\alpha^+}(c_\alpha)$ is trivial.\label{item:abelian}
        \item\label{item:finite_index} $\langle c_{\alpha}, o_\alpha \rangle$ has finite index in $\tau_{\alpha}(G_{\alpha})$. We denote by $H_{\alpha}$ a finite set of coset representatives of cosets of $\langle c_{\alpha}, o_\alpha \rangle$.
        \item\label{item:no_name} $t_{\alpha} c_{\alpha}t_{\alpha}^{-1} = o_{\bar{\alpha}}$ (recall the generators $t_\alpha$ from the definition of $\pi_1(\mathcal G)$).\label{prop:conjugate_by_edge}
        \item If $\alpha^+ = \beta^+$, then $c_\alpha = c_{\beta}$.\label{prop:same_center}
        \item If $wo_{\alpha}^khw'\in \tau_{\alpha}(G_{\alpha})$ for some $h\in H_{\alpha}$ and some elements $w, w'\in G_{\alpha^+}$, then either both $w$ and $w'$ are in $\tau_{\alpha}(G_{\alpha})$ or $\abs{w}+ \abs{w'}\geq k/K$. \label{prop:length_bound_for_edge_group}
        \item The word $c_\alpha^ko_{\alpha}^\ell h$ is a $K$-quasi-geodesic in the given word metric on $\pi_1(\mathcal G)$ for all $k, \ell\in \Z$ and $h\in H_{\alpha}$.\label{prop:quasi-geodesic}
    \end{enumerate}
\end{restatable}

\begin{proof}
We let $c_\alpha$ be a generator of the centre $Z_{\alpha^+}$ of the vertex group $G_{\alpha^+}$, which lies in $\tau_\alpha(G_\alpha)$ by Definition \ref{def:admissible}-\eqref{item:edge_groups}. We let $o_\alpha$ be $\tau_\alpha(\tau_{\bar\alpha^{-1}}(c_{\bar\alpha}))$ (which should be interpreted as $c_{\bar\alpha}$ ``transported'' across $\alpha$ from $\alpha^-$ to $\alpha^+$).

    \eqref{item:abelian} By Definition \ref{def:admissible}-\eqref{item:different_R}, we have that $c_\alpha$ and $o_\alpha$ generate subgroups of $\tau_{\alpha}(G_{\alpha})$ that intersect trivially (since $\tau_\alpha$ is injective), and in particular $c_{\alpha}^ko_{\alpha}^\ell$ can be trivial only if $k=\ell=0$. By construction $\pi_{\alpha^+}(c_\alpha)$ is trivial.

    \eqref{item:finite_index} We have that $\langle c_{\alpha}, o_\alpha \rangle<\tau_{\alpha}(G_{\alpha})$ is isomorphic to $\mathbb Z^2$, since $c_\alpha$ and $o_\alpha$ have infinite order and they generate subgroups of an abelian group that intersect trivially. Since $\tau_{\alpha}(G_{\alpha})$ is also isomorphic to $\mathbb Z^2$, we must have that $\langle c_{\alpha}, o_\alpha \rangle$ has finite index.

\eqref{item:no_name} This follows from the construction of $o_\alpha$ and the third type of relation for $\pi_1(\mathcal G)$.

    \eqref{prop:same_center} This holds by construction of the $c_\alpha$.

    \eqref{prop:length_bound_for_edge_group} Fix the notation of the statement. Note that $g=o_\alpha^kh\in \tau_{\alpha}(G_\alpha)$. In particular, we cannot have that exactly one of $w$ and $w'$ lies in $\tau_\alpha(G_\alpha)$, for otherwise $wgw'$ is a product of two elements of $\tau_\alpha(G_\alpha)$ and one element outside of $\tau_\alpha(G_\alpha)$, so it cannot lie in $\tau_\alpha(G_\alpha)$. Suppose then that both $w$ and $w'$ do not lie in $\tau_\alpha(G_\alpha)$. Let now $\rho:G_{\alpha^+}\to \langle w a_\alpha \rangle$ be the composition of $\pi_{\alpha^+}$ (the homomorphism onto the hyperbolic group $F_{\alpha^+}$) and the closest-point projection to $w \langle a_\alpha \rangle$ (with respect to a fixed word metric). Recall that $\langle a_\alpha \rangle$ is malnormal in $F_{\alpha^+}$ by Definition \ref{def:admissible}-\eqref{item:malnormal}. Since it is cyclic, it is also quasiconvex, and in particular $F_{\alpha^+}$ is hyperbolic relative to $\langle a_\alpha \rangle$ \cite[Theorem 7.11]{Bow:rel_hyp}. 
    Recall that closest-point projections onto cosets of peripherals are coarsely Lipschitz, and that distinct cosets of peripherals have uniformly bounded closest-point projection onto each other, see \cite{S:proj}.
    Since $\pi_{\alpha^+}(w)\notin \langle a_\alpha \rangle$ we then have that $\rho(\tau_\alpha(G_\alpha))$ is contained in the $L\abs{w}$-neighborhood of $\pi_{\alpha^+}(w)$, for some constant $L$ depending only on $G_{\alpha^+}$. Similarly, $\rho(\tau_\alpha(wgw'G_\alpha))$ is contained in the $L\abs{w'}$-neighborhood of $\pi_{\alpha^+}(wg)$. But we are assuming $\tau_\alpha(wgw'G_\alpha)=\tau_\alpha(G_\alpha)$, so in order to conclude, we need to argue that $d(\pi_{\alpha^+}(wg), \pi_{\alpha^+}(w))=d(\pi_{\alpha^+}(q), \pi_{\alpha^+}(1))$ is at least $k/L'$ for some uniform $L'$. This follows from the fact that $o_\alpha$ maps to an infinite order element of the hyperbolic group $F_{\alpha^+}$, and therefore generates an undistorted infinite cyclic subgroup of $F_{\alpha^+}$, concluding the proof.

    \eqref{prop:quasi-geodesic} This follows from the fact that $\tau_\alpha(G_\alpha)$ is undistorted in $\pi_1(\mathcal G)$, see e.g. \cite[Corollary 2.16]{FK:connected}.
\end{proof}

\section{Auxiliary algorithms and word problem for 3-manifold groups}\label{sec:more_alg}

In this section we collect various auxiliary algorithms needed to prove Theorem \ref{thm:3_mflds} on 3-manifold groups and Theorem \ref{thm:admissible} on admissible graphs of groups. At the end we will also explain how to prove Theorem \ref{thm:3_mflds} using Theorem \ref{thm:admissible}, known results from the literature, and auxiliary algorithms.

We start with the following, which is well-known, but we include a proof for completeness.

\begin{lem}
\label{lem:finite_index}
    Let $H<G$ be a subgroup of finite index and let $f(n)\geq n$ be a function. If the word problem in $H$ can be solved in time at most $f(n)$, then the word problem in $G$ can also be solved in time at most $f(n)$. 
\end{lem}

\begin{proof}
    Fix a generating set $S$ for $G$ and a set $\mc C$ of right coset representatives for $H$, where $1$ is the representative for the coset $H$. It is well-known (and follows from arguments below) that $H$ is generated by $T=\{c_1sc_2\in H| c_i\in\mc C, s\in S\}$. We now describe a Turing machine $\mc M$ (with two tapes) which takes as input a word $w$ of length $n$ in $S$ and determines in linear time whether $w$ represents an element of $H$, and if so outputs a word $w'$ of length $n$ in $T$ representing the same element. This suffices to prove the proposition.

    The Turing machine $\mc M$ reads the input word $w=s_1\ldots s_n$, and inductively at step $k$ has kept track (in its state) of the unique coset representative $c_k\in\mc C$ such that $s_1\ldots s_k\in H c_k$, and has outputted (on the second tape) a word $w_k$ of length $k$ in $T$ that represents $s_1\ldots s_kc^{-1}_k$.
    
    If $k<n$, then $\mc M$ reads $s_{k+1}$ and updates $c_k$ to the unique coset representative $c_{k+1}\in \mc C$ of $Hc_ks_{k+1}$. In particular, we have that $s_1\ldots s_k s_{k+1}\in H c_{k+1}$. Moreover, $\mc M$ appends $c_ks_kc_{k+1}^{-1}$, as an element of $T$, to $w_k$ to form the word $w_{k+1}$ representing $s_1\ldots s_{k+1}c_{k+1}^{-1}$. 

    If at the end $c_n = 1$, then $w$ is in $H$ and $w_n$ is the desired word $w'$. Otherwise, $w$ is not in $H$ and hence not the trivial word.
\end{proof}

Next, we study central extensions of hyperbolic groups, which is useful for us since these are vertex groups of admissible graphs of groups, and also some 3-manifolds groups are of this form.

\begin{prop}
\label{prop:central_word_problem}
    Let $G$ be a $\mathbb Z$-central extension of a hyperbolic group, that is, $G$ is a group satisfying $\Z\hookrightarrow G{\twoheadrightarrow} H$ for some hyperbolic group $H$. Then there is an algorithm that takes as input a word in the generating set and determines whether it represents an element of (the image of) $\mathbb Z$, and if so which one, and which runs in time linear in the length of the input word.
\end{prop}

\begin{proof}
    We can take $G$ to be, as a set, $H\times \mathbb Z$, with the operation $(h,n)(h',m)=(hh',n+m+\sigma(h,h'))$ for some cocycle $\sigma:H\times H\to \mathbb Z$.
    We can also fix a generating set $T$ for $G$ consisting of $(e,\pm1)$, where $e$ is the identity of $H$, and $S\times\{0\}$ for a finite symmetric generating set $S$ for $H$. We denote by $\pi$ the map from words in $T$ to words in $S\cup\{e\}$ coming from projection on the first coordinate, and with a slight abuse of notation we will sometimes conflate $\pi(w)$ with the element of $H$ that it represents. Given a word $w$ in $T$, there is $n_w\in \mathbb Z$ such that $w$ represents $(\pi(w),n_w)$ in $G$. Finally, let $\delta\geq 1$ be a hyperbolicity constant for the Cayley graph of $H$ with respect to $S$, say with $\delta$ an integer.

    It is well-known that, given a hyperbolic group, Dehn's algorithm can be run in linear time on a Turing machine, and in particular the following algorithm can also be implemented on a Turing machine in linear time in the length of the input word; we explain this in more detail below.
    
    \begin{enumerate}
        \item Start with an input word $w$ in $T$, and let $w_0$ be the word obtained removing from $w$ all letters of the form $(e,\pm 1)$. Let $n_0$ be the sum of the $\pm 1$ appearing in the second coordinates of such letters. Note that the element of $G$ represented by $w$ is $(\pi(w_0),n_{w_0}+n_0)$.
        
        At each step we will have a word $w_i$ in $S\times\{0\}$ and an integer $n_i$ such that $(\pi(w_i),n_{w_i}+n_i)$ represents the same element of $G$ as $w$.
        
        \item At each step $i$, exactly one of the following occurs.
        \begin{enumerate}[label=(\alph*)]
            \item $w_i$ is empty. In this case $w$ represents an element of $\mathbb Z$, the element being $n_i$ since $n_{\emptyset}=0$, and the algorithm stops.
            \item the previous case does not occur, and the word $\pi(w_i)$ is a $9\delta$-local geodesic (meaning that all subwords of length at most $9\delta$ represents geodesics in the fixed Cayley graph of $H$). In this case it is well known (see e.g. \cite[Theorem III.H.1.13]{BridsonHaefliger}) that by hyperbolicity $\pi(w_i)$ cannot represent a loop in the Cayley graph of $H$, that is, $\pi(w_i)\neq e$, and $w$ does not represent an element of $\mathbb Z$. The algorithm stops.
            \item the word $\pi(w_i)$ is not a $9\delta$-local geodesic, so that $w_i=abc$ where $b$ is a word of length at most $9\delta$ such that the word $\pi(b)$ is not geodesic. Let $b'$ be a word in $S\times \{0\}$ such that $\pi(b')$ is a geodesic word representing the same element of $H$ as $b$ (in particular $b'$ is shorter than $b$). We can set $w_{i+1}=ab'c$ and $n_{i+1}=n_i+n_{b}-n_{b'}$. Adding $n_{b}-n_{b'}$ at the various steps does not change the running time as this quantity is uniformly bounded and there are linearly many steps. We verify below that  $(\pi(w_{i+1}),n_{w_{i+1}}+n_{i+1})$ represents the same element of $G$ as $w$.
        \end{enumerate} 
    \end{enumerate}
    
In the notation of the last case, we are only left to show that $(\pi(abc),n_{abc}+n_i)=(\pi(ab'c),n_{ab'c}+n_i+n_{b}-n_{b'})$, that is, that $n_{abc}=n_{ab'c}+n_{b}-n_{b'}$.
Using the definition of the operation in $G$, we then have that the element of $G$ represented by $abc$ is
$$(\pi(abc),n_a+n_b+n_c+\sigma(\pi(a),\pi(b))+\sigma(\pi(ab),\pi(c))),$$
so that the second coordinate is $n_{abc}$.
Similarly, the element of $G$ represented by $ab'c$ is
$$(\pi(ab'c),n_a+n_{b'}+n_c+\sigma(\pi(a),\pi(b'))+\sigma(\pi(ab'),\pi(c))).$$
But in $H$ we have $\pi(b)=\pi(b')$ and $\pi(ab)=\pi(ab')$, so that the expressions for $n_{abc}$ and $n_{ab'c}$ have the same terms except for $n_{b}$ and $n_{b'}$, yielding required equality.

As promised, we explain in more detail how to implement the algorithm in linear time on a Turing machine with three bands, where the third band stores the $n_i$, and we disregard it below as, explained above, it is straightforward to update it at each step. Specifically, we have to implement the algorithm in a way that step 2(c) is performed at most linearly many times, and a bounded number of operations is performed when passing from $w_i$ to $w_{i+1}$. To ensure this, we inductively arrange the following. The word $w_i$ has length at most $n$ and it has an initial subword $a_i$ stored in one band, while the rest of the word is stored on a different band, and the two heads are respectively at the end and beginning of the stored words. The word $a_i$ is a $9\delta$-local geodesic, and $ 9\delta|w_i|-|a_i|\leq 9\delta n-i$. 

If $a_i=w_i$ either 2(a) or 2(b) applies, and indeed we performed a linear number of operations. If not, we can read the subword $u$ of $w_i$ consisting of the last $\min\{|a_i|,9\delta-1\}$ letters of $a_i$ and the letter of $w$ coming after $a_i$. If $u$ is a geodesic, then the subword $a_{i+1}$ of $w_i$ obtained by adding a letter at the end of $a_i$ is a $9\delta$-local geodesic, and the relevant head is at the end of it, so we can set $w_{i+1}=w_i$. When replacing $i$ with $i+1$, the quantity $9\delta|w_i|-|a_i|$ decreases by 1, so the induction hypothesis holds. If $u$ is not a geodesic, we can replace $w_i$ with $w_{i+1}$ as in 2(c) with $b = u$ and $a$ as the subword of $w$ coming before $u$, and where $b'$ gets stored in the second band. We then move the relevant head of the Turing machine to the end of $a$, which we set equal to $a_{i+1}$. When replacing $i$ with $i+1$, the quantity $|a_i|$ decreases by at most $9\delta-1$, while the quantity $9\delta|w_i|$ decreases by at least $9\delta$, so overall $9\delta|w_i|-|a_i|$ decreases by at least 1, and the induction hypothesis holds.
\end{proof}

Finally, we will need that the membership problem for quasiconvex subgroups of hyperbolic groups can be solved in linear time. This is well-known, but we include a proof to explicitly point out that one can further require the output to include a quasi-geodesic representative. The proposition can also be extracted with some work from \cite{KMW}, but here we explain a more direct approach

\begin{prop}
\label{prop:quasi_convex}
    Let $H$ be a quasiconvex subgroup of a hyperbolic group $G$, with a fixed generating set $S$ containing a generating set $T$ for $G$. Then there exists a constant $\lambda\geq 1$ and an algorithm that runs in $O(n)$ that determines whether a word $w$ in $S$ represents an element of $H$, and if so it outputs a $(\lambda,\lambda)$-quasi-geodesic representative in the alphabet $T$ of the element of $H$ represented by $w$.
\end{prop}

\begin{proof}
    Fix a finite symmetric generating set $S$ for $G$, which we regard as endowed with the corresponding word metric. It is well-known, via Dehn's algorithm, that there exists $\lambda\geq 1$ and an algorithm that given as input a word of length $n$ in $S$ returns in time $O(n)$ another word $w'$ in $S$ representing the same element of $G$, with the word $w'$ being $(\lambda,\lambda)$-quasi-geodesic. Therefore, it suffices to describe an algorithm as in the statement except that the input word is required to be $(\lambda,\lambda)$-quasi-geodesic. For short, we call such word a $\lambda$-word. We now show that there exists a finite-state automaton that accepts a $\lambda$-word $w$ if and only if $w$ represents an element of $H$, which yields the required linear-time algorithm.

    There exists a constant $M\geq 0$ such that any $(\lambda,\lambda)$-quasi-geodesic with endpoints on $H$ stays within $M$ of $H$. The state of the finite-state automaton are elements $x$ of the ball of radius $M$ in $G$ such that $1$ is a closest element of $H$ to $x$. The initial state, and only final state, is the identity $1$. Consider now any $x$ as above and some $s\in S$ such that $xs$ still lies in the $M$-neighborhood of $H$. Let $h=h(x,s)$ be closest to $xs$ in $H$. We now connect $x$ to $h^{-1}xs$ (as states of the finite-state automaton). 

Now, if a $\lambda$-word represents an element of $H$ then it stays in the $M$-neighborhood of $H$, which in turn implies that it gives a path in the finite-state automaton. It is then readily seen that such path ends at the final state.

Conversely, if a $\lambda$-word $s_1\dots s_n$ ends at the final state, then it represents the element $\prod_{i=1}^n h(x_{i-1},s_i)\in H$ where $x_0=1$ and $x_{i}=h(x_{i-1},s_i)^{-1}x_{i-1}s_i$. Indeed, it is straightforward to prove by induction that $s_1\dots s_j=(\prod_{i=1}^j h(x_{i-1},s_i))x_j$.

The finite-state automaton can be implemented on a Turing machine, and we can add an additional band that writes as output a concatenation of geodesic words in $T$ representing the $h(x_{i-1},s_i)$; up to increasing $\lambda$ this is a $(\lambda,\lambda)$--quasi-geodesic word.
\end{proof}

We are now ready to prove our main theorem on the word problem for 3-manifold groups, assuming Theorem \ref{thm:admissible} and Proposition~\ref{prop:free_product}. The proof of the latter, in particular Theorem~\ref{thm:admissible}, is the focus of the rest of this paper.

\begin{proof}[Proof of Theorem \ref{thm:3_mflds}]
    We will use the geometrisation theorem. First of all, $M$ is a connected sum of prime 3-manifolds $M_i$, so that $\pi_1(M)$ is isomorphic to the free product of the $\pi_1(M_i)$. In view of Proposition \ref{prop:free_product}, it suffices to show that if $M$ is prime then the word problem can be solved in $O(n\log^2(n))$.

    If $M$ is prime it can be either geometric or non-geometric. If it is geometric, then it has one of the 8 geometries below, and in each case we explain why the word problem is solvable in $O(n\log^2(n))$. See \cite{Scott} for more information on the geometries.

    \begin{itemize}
        \item $S^3$. In this case $\pi_1(M)$ is finite.
\item $\mathbb R^3$. In this case $\pi_1(M)$ is virtually abelian. The word problem is solvable in linear time in abelian groups, and the word problem in a finite-index subgroup has the same complexity as the word problem in the ambient group as argued in Lemma \ref{lem:finite_index}.
\item $\mathbb H^3$. In this case $\pi_1(M)$ is hyperbolic and hence has word problem solvable in linear time, see e.g. \cite[Theorem 2.18]{notes}.
\item $S^2\times \mathbb R$. In this case $\pi_1(M)$ is virtually cyclic, and in particular virtually abelian.
\item $\mathbb H^2\times \mathbb R$. In this case $\pi_1(M)$ is a $\mathbb Z$-central extension of a hyperbolic group, so we can use Proposition \ref{prop:central_word_problem} to find a linear-time solution to the word problem.
\item $\widetilde{PSL}_2(\mathbb R)$. In this case $\pi_1(M)$ is again a $\mathbb Z$-central extension of a hyperbolic group.
\item $Nil$. Nilpotent groups have word problem which is solvable even in real-time \cite{real_time_wp}.
\item $Sol.$ For this case, see \cite[Proposition 2]{average_case}; the word problem is solvable in $O(n\log^2 n)$.
    \end{itemize}

If $M$ is non-geometric, then it is either a graph manifold or it contains a hyperbolic component in its geometric decomposition. In the first case, $\pi_1(M)$ is the fundamental group of an admissible graph of groups, so we can use Theorem \ref{thm:admissible}. In the latter case, according to \cite{mixed_virt_spec} $\pi_1(M)$ virtually embeds in a right-angled Artin group. By Lemma \ref{lem:finite_index} the ``virtually'' does not affect the complexity of the word problem, and \cite{wp_raags,wp_monoids} provide a linear-time solution to the word problem in right-angled Artin groups; the algorithm is also explained in \cite[pages 443-444]{conj_raags}. Both papers use the RAM model of computation, but it is readily seen that the algorithm can be implemented in linear time also on a Turing machine (with one tape for each generator keeping track of the stack of the piling which corresponds to it).

This was the last case, and the proof is complete.
\end{proof}

\section{A perspective on solving the word problem for free products}\label{sec:free_prod}

In this section, we aim to highlight some difficulties about solving the word problem for free products and give some intuition on what an efficient solution of the word problem for free products (and, by extension, for more general graph of groups like admissible graph of groups) should look like. In particular, we want to highlight why the ``straightforward'' approaches are slow. We then use the gained knowledge to describe what we want from an efficient algorithm. This gives an intuition why the middle derivations (which is the key tool in our algorithm for free products and admissible graph of groups) as described in the next sections are defined the way they are. 

In the following, $G = A\ast B$ is a free product and we assume that the word problem in $A$ and $B$ can be solved in time at most $f(n)$ time for a convex function $f(n)\geq n$. Further, $S_A, S_B$ are symmetric generating sets of $A$ and $B$ respectively which do not contain the identity. In particular, any (nontrivial) word $w$ over $S = S_A\cup S_B$ can be written as $w = a_1b_1\ldots a_kb_k$ where the $a_i$ are words over $S_A$ (and nonempty for $i\neq 1$) and the $b_i$ are words over $S_B$ (and nonempty for $i\neq k$). We call the $a_i$ and $b_i$ \emph{syllables} and say that the $a_i$ are of ``type $A$'' and the $b_i$ are of ``type $B$''. Note that since $G$ is a free product, such a partition is in fact unique. 

\begin{rem}\label{rem:triviality}
    If all words $a_i$ and $b_i$ (except potentially $a_1$ and $b_k$, but not both if $k = 1$) do not represent the trivial element, then $w$ does not represent the trivial element.
\end{rem}

In light of Remark \ref{rem:triviality}, it makes sense to figure out which of the syllables are trivial and then remove them from the word. It turns out that the order in which we do this has a great influence on the runtime of the algorithm.

\subsection{Many iterations}

One way to check whether a word $w$ is trivial is the following. We call this the $\emph{many iterations}$ way.

\begin{enumerate}
    \item Write your word $w$ as $a_1b_1\ldots a_kb_k$ as described above.
    \item Go through the syllables $a_i$ and $b_i$ from left to right and check whether they are trivial. Since we assumed we can solve the word problem in each factor in time at most $f(n)$, this will take at most $f(\abs{a_i})$ (resp. $f(\abs{b_i})$) time. If it turns out that the current syllable does not represent the trivial word, append it to a new (originally trivial) word $w'$. 
    \item Check whether $w'$ is empty. If it is empty, then $w'$ (an hence $w$) represents the trivial word. If $w'$ has the same number of syllables as $w$, then $w$ is not the trivial word (see Remark \ref{rem:triviality}). If neither of the above happened, redefine $w$ as $w'$ and $w'$ as the trivial word and go back to Step 1.
\end{enumerate}

\begin{ex}
    Take $A = \langle a \rangle$ and $B = \langle b \rangle$ and let us examine what the many iterations algorithm does for the word $w = a b b^{-1} a b$. 

    We start by writing $w = a_1b_1a_2b_2$, where 
    \begin{align*}
        a_1 = a, b_1 = bb^{-1}, a_2 = a, b_2 = b
    \end{align*}

    The only trivial syllable is $b_1$, so after the first iteration, we will have $w' = a_1a_2b_2 = aab$. We then go back to Step 1, and write $aab = a_1b_1$, where $a_1 = aa$ and $b_1 = b$. We see that neither of the syllables are trivial, so at the end $w'$ will have the same amount of syllables as (the current word) $w$. Hence we know that the word is not trivial and we stop. 
\end{ex}

\begin{ex}\label{ex:slow-many}
    Next we will examine an example where the many iterations algorithm takes $\succeq n^2$ time. Again, we take $A = \langle a \rangle$ and $B = \langle b \rangle$. Consider
    \begin{align*}
        w = \underbrace{ab\ldots ab}_{2k} \underbrace{b^{-1}a^{-1}\ldots b^{-1}a^{-1}}_{2k}
    \end{align*}
    and let $n = \abs{w} = 4k$. Note that $w = a_1b_1\ldots a_{2k}b_{2k}$, where 
    \begin{align}
        a_i = \begin{cases}
            a &\text{if $i\leq k$,}\\
            a^{-1} & \text{otherwise,}\\
        \end{cases}
        \qquad b_i = \begin{cases}
            b & \text{if $i<k$,}\\
            bb^{-1} & \text{if i = k,}\\
            b^{-1} & \text{if $k<i<2k$,}\\
            \text{trivial} & \text{if i = 2k.}\\
        \end{cases}
    \end{align}
    In particular, the only syllables representing the trivial word are $b_k$ (and $b_{2k}$). Hence when we arrive at Step 3 the first time, we have that 
    \begin{align*}
        w' = \underbrace{ab\ldots a}_{2k-1} \underbrace{a^{-1}\ldots b^{-1}a^{-1}}_{2k-1}.
    \end{align*}
    Since we had to check every syllable, the first iteration of Steps 1 to 3 took $\succeq n f(1)\succeq n$ time. For the next iteration, we again only have two syllables (namely $a_{k-1}$ and $b_{2k-1}$) which are trivial, so when we reach Step 3 we have
    \begin{align*}
        w' = \underbrace{ab\ldots ab}_{2k-1} \underbrace{b^{-1}a^{-1}\ldots b^{-1}a^{-1}}_{2k-1} 
    \end{align*}
    and performing Step 2 took $\succeq n - 2$ time. In total, it will take us $2k$ iterations until we determine that the starting word was representing the trivial element and a total of $\succeq n + (n-2) +\ldots 4 + 2 \succeq n^2$ time.
\end{ex}

The above example shows that the many iterations algorithm is too slow. One takeaway from the algorithm is the following: when we discovered that $b_k$ was trivial, instead of checking whether $a_{k+1}$ is trivial, we should have checked whether $a_{k}a_{k+1}$ is trivial. When we realized that it was trivial, we should have checked $b_{k-1}b_{k+1}$ and so on. The idea to check syllables in that order is described in the next section.

\subsection{One iteration}

We now incorporate the idea described in the paragraph above in an algorithm we call \emph{one iteration}. For this, $w'$ will be an auxiliary word of the form $w' = c_1'\ldots c_\ell'$, where the $c_i'$ are the nonempty syllables of $w'$. The algorithm works as follows.

\begin{enumerate}
    \item Write $w = a_1b_1\ldots a_kb_k = c_1\ldots c_{2k}$, where $c_{2i-1} = a_i$  and $c_{2i} = b_i$. Start with $w'$ as the empty word.
    \item Go through the syllables $c_i$ from left to right.
    \item If the last syllable of $w'$ is of the same type (this requires $w'$ to not be the trivial word) as $c_i$, then do the following.
    \begin{itemize}
        \item Let $c_\ell'$ be the last (nonempty) syllable of $w'$. Check whether $c_\ell'c_i$ represents the trivial word. If so, remove $c_\ell$ from $w'$. If not, replace $c_\ell'$ with the syllable $c_\ell'c_i$.
    \end{itemize}
    \item If $w'$ is empty or the last syllable of $w'$ is not the same type as $c_i$, then do the following. 
    \begin{itemize}
        \item Check whether $c_i$ represents the trivial word. If not, add the syllable $c_i$ to $w'$. 
    \end{itemize}
\end{enumerate}

Note that at the end, all syllables of $w'$ are nonempty by design, so we know that $w$ represents the trivial word if and only if $w'$ is empty. Thus, already after ``one iteration'' we know whether $w$ represents the trivial word.

\begin{ex}
    Consider again the word from Example \ref{ex:slow-many}. That is $A = \langle a\rangle$, $B = \langle b \rangle$ and 
    
    \begin{align}
        w = \underbrace{ab\ldots ab}_{2k} \underbrace{b^{-1}a^{-1}\ldots b^{-1}a^{-1}}_{2k}.
    \end{align}

    We start with $w'$ as the empty word. Then for $i=1$, we see that $w'$ is empty, so we check whether $a_1 = a$ represents the trivial word, which it does not. Thus we add $a$ to $w'$. Then for $i=2$ we check whether the last syllable of $w'$ (which is $a$) has the same type as $c_2 = b_1 = b$, which it does not. Hence we check whether $c_2$ represents the trivial word (which it does not) and we add $b$ to $w'$. After $2k-1$ steps, we have that 
    \begin{align*}
        w' = \underbrace{aba..\ldots a}_{2k-1}.
    \end{align*}
    Now we consider $c_{2k} = bb^{-1}$. It represents the trivial word (and has a different type then the last syllable of $w'$). Hence we do not change $w'$ and go to the next step. There we see that $c_{2k+1} = a^{-1}$ has the same type as the last syllable of $w'$. Further $aa^{-1}$ represents the trivial word, so we delete $a$ from $w'$. At the end we have that $w'$ is the trivial word and hence we know that $w$ represents the trivial word. In total, this will run in time at most $ n $.
\end{ex}

The example above gives hope that the ``one iteration'' algorithm has a good runtime. However, the next example shows that this is not the case. 

\begin{ex}
    Consider $B = \langle b \rangle$, $A$ a group whose word problem is solvable in time $f$ and let $a_1 ,\ldots, a_k$ be words in $A$ such that 
    \begin{itemize}
        \item $\abs{a_1}= k$, 
        \item $\abs{a_i} = 1$ for $2\leq i \leq k$,
        \item For all $i$, $\prod_{j=1}^i a_j$ does not represent the trivial word.
    \end{itemize}
    Consider $w = a_1 bb^{-1} a_2 bb^{-1}a_3 \ldots bb^{-1}a_k$. We have that $n = \abs{w} = 4k-3$. When we perform the algorithm, after the first step we have $w' = a_1$ and we have to check whether $a_1$ is trivial, which takes up to $f(k)$ time. In the third step we have to check whether $a_1a_2$ is trivial, which takes up to $f(k)$ time. Since it is not trivial, we have afterwards that $w' = a_1a_2$. Continuing this, we see that for odd $i$, the $i$-th step takes up to $f(k)$ time. Thus in total, the algorithm takes up to $kf(k)\succeq nf(n)$ time.
\end{ex}

The example above shows that the ``one iteration'' algorithm is slow as well. 

\subsection{Analysing the reasons for the slow runtime}

The ``one iteration'' algorithm is slow because it (potentially) checks over and over again whether a certain subword of $w'$ plus a small subword of $w$ is the trivial word. If the last syllable of $w'$ is long compared to the added syllables, then we rack up a lot of runtime. On the other hand, if we only check the added syllables and ignore $w'$ (as is done in ``many iterations'') then we might have to do a linear amount of iterations until we reach the trivial word (as in Example \ref{ex:slow-many}), which again blows up a runtime. Thus, to get an efficient runtime, we need to strike a balance. On one hand, we want to guarantee that there are not too many iterations (say at most $\log (n)$) and on the other hand, we want to not check long syllables in $w'$ too many times, especially when the current syllable from $w$ is short. 

\section{Solving the word problem for admissible graph of groups}
\label{sec:admissible}

In this section we solve the word problem for admissible graphs of groups, as well as free products.

\par\smallskip

\textbf{Setup:} In this section $\mc G$ denotes a graph of groups with underlying graph $\Gamma$.

\subsection{Organized words and derivations}

In this subsection we discuss reductions for words representing elements of $\pi_1(\mathcal G)$, and, most importantly, we introduce derivations of words. These are words that represent the same group elements of $\pi_1(\mathcal G)$ as a given word and will play a crucial role.

Fix a base vertex $v_0\in V(\Gamma)$. We assume that $S$ is a generating set of $G = \pi_1(\mc G,v_0)=\pi_1(\mc G)$ satisfying 
\begin{align*}
    S =( \cup_{v\in V(\Gamma)} S_{v})\cup (\cup_{\alpha\in E(\Gamma)} S_\alpha )\cup E(\Gamma) \cup \{\alpha_\ast\},
\end{align*}
where $S_{v}$ and $S_{\alpha}$ are symmetric generating sets of $G_v$ and $\tau_\alpha (G_\alpha)$ respectively. Further, if $\mc G$ is admissible, we require that $S_{\alpha}\subset S_{\alpha^+}$ contains $o_{\alpha}, c_{\alpha}$ and $H_{\alpha}$ as in Lemma~\ref{lem:properties_of_vertex_groups}. The generator $\alpha_\ast$ is a dummy generator which we require to satisfy $\alpha_\ast =_G 1$ and should be thought of as an additional edge from $v_0$ to $v_0$ with formal inverse $\bar{\alpha}_\ast$ and which satisfies $\alpha_\ast^+ = \alpha_\ast^- = v_0$.

\begin{defn}[organized]
    We say a word $w$ is \emph{organized} if $w = t_0w_1t_1w_2t_2\ldots w_nt_n$ for some $n\geq 1$ and the following hold.
    \begin{enumerate}
        \item ($t_i$ are edges) $t_i\in E(\Gamma)$ for all $1\leq i < n$.\label{cond:t_i_are_edges}
        \item (corners are special)\label{cond:extra_edges}
        \begin{enumerate}
            \item $t_0 =  \alpha_\ast$\label{cond:extra_edges0}
            \item $t_n = \alpha_\ast$\label{cond:extra_edgesn}
        \end{enumerate}
        \item (path in $\Gamma$) $t_{i-1}^+ = t_{i}^-$ for all $1\leq i \leq n$\label{cond:bulid_path}
        \item $w_i$ is a (possibly empty) word in $S_{t_{i}^-} = S_{t_{i-1}^+}$.\label{cond:vertex_groups}
    \end{enumerate}
    We say that $n$ is the \emph{path-length} of $w$ and denote it by $\pabs{w}$. We call the words $u_i(w):=w_i$ \emph{pieces} of $w$. Given a piece $u_i(w)$ of $w$ we say that $t_{i-1}$ is its \emph{left edge} and $t_i$ is its \emph{right edge}. It will be useful to also introduce the notation $t_i(w):=t_i$.

\end{defn}
The Conditions \eqref{cond:t_i_are_edges}-\eqref{cond:bulid_path} imply that $p = (t_1,t_2,\ldots,t_{n-1})$ is an edge path in $\Gamma$ starting and ending at $v_0$. We should think of $w$ as follows: while going along the path $p$, we travel along words $w_i$ which are in the current vertex group. 

A prefix $w'$ of a organized word $w$ is called \emph{semi-organized}. We denote the last edge (piece) of a semi-organized word $w'$ by $t_{last}(w')$ ($u_{last}(w')$). Note that the last letter of a semi-organized word does not have to be an edge, that is, it does not have to be in $E(\Gamma)\cup \{\alpha_\ast\}$. If the last letter of $w'$ is an edge then we say that $w'$ is \emph{full}. Observe that full semi-organized words satisfy all conditions for organized words except \eqref{cond:extra_edgesn}.

\begin{defn}
    Let $w$ be an organized word. We say that a piece $u_i(w)$ is \emph{$w$--reducible} if $t_{i-1}(w) = \bar{t}_i(w)$ and $u_i(w)\in \tau_{t_{i-1}}(G_{t_{i-1}})$. If a piece is not $w$--reducible we call it \emph{$w$--irreducible}.
\end{defn}

\begin{defn}
    We say that an organized word $w$ is \emph{shortened} if all its pieces are $w$--irreducible. (This is often called reduced, but we chose a term less similar to ``reducible''.)
\end{defn}

The following is a basic, well-known result about graphs of groups which says, roughly, that a word represents the trivial element if and only if it ``visually'' does so.

\begin{lem}\label{lemma:tirivality-cirteria-shortened}
   If $w$ is a shortened word, then $w$ represents the trivial element if and only if $\pabs{w} = 1$ and $u_1(w)$ represents the trivial word in $G_{v_0}$.
\end{lem}

Finally, we introduce derivations of words, which, roughly, are obtained from a given word by replacing subwords that represent elements of a vertex group by a word in the generating set of the vertex group.

\begin{defn}\label{def:derivation}
    We say that an organized word $w'$ of path-length $\pabs{w'} = n$ is \emph{derived} from an organized word $w$ of path-length $\pabs{w} = m$, in symbols $w'\prec w$, if there exists a sequence $k_0< k_1<k_2< \ldots < k_n$ such that $k_0 = 0, k_n = m$ and the following hold.
    \begin{enumerate}
        \item $t_i(w')= t_{k_i}(w)$ for all $0\leq i \leq n$\label{der1}
        \item $u_i(w') =_{G} \left( \prod_{j= k_{i-1}+1}^{k_i - 1}u_{j}(w) t_{j}(w)  \right)u_{k_{i}}(w) =: r_i$ for all $1\leq i \leq n$. \label{der2}
    \end{enumerate}
    We call the sequence $\mathfrak k = (k_0, \ldots, k_n)$ a split-sequence. We say that $u_i(w')$ is derived from $r_i$ and $r_i$ is the \emph{$(w, \mathfrak k)$--base} of $u_i(w')$. We say that the $(w, \mathfrak k)$--derived length of $u_i(w')$, denoted by $\norm{u_i(w')}_{w, \mathfrak k}$, is $\abs{r_i}+1$. 
    For any piece $u_j(w)$ with $k_{i-1}+1\leq j \leq k_i$ we say that $u_i(w')$ is its \emph{$(w', \mathfrak k)$--extension}.
\end{defn}

 Note that the derived length, base, and extensions all depend on the split-sequence $\mathfrak k$. Further, there might be multiple split-sequences $\mathfrak k$ which satisfy \ref{der1} and \ref{der2}. However, in the following section, it will be often clear which split sequence we refer to. Thus, by abuse of notation, when it is clear which split sequence is used for the derivation or it does not matter which one we use, then we will omit it from the notation. For example, we would write $\norm{u_i(w')}_{w}$ instead of $\norm{u_i(w')}_{w, \mathfrak k}$. In particular, if we also have a word $w''$ which is derived from $w'$ (with split-sequence $\mathfrak k' = (k_0', \ldots, k'_{\pabs{w''}})$ ), then, $w''$ is derived from $w$ and the canonical the associated split-sequence is $\mathfrak k'' = (k_0'', \ldots, k_{\pabs{w''}}'')$, where $k_i'' = k_{k_i'}$ for all $i$.

When we construct derivations they usually have the following elementary form. Given an organized word $w$ with $n=\pabs{w}$ and a piece $u_i(w)$ which is reducible, define
\[
w' = \left (\prod_{j=1}^{i-1}t_{j-1}(w)u_j(w)\right ) x \left(\prod_{j=i+1}^{n}u_j(w)t_j(w)\right),
\]
where $x$ is a word in $S_{t_{i-2}(w)^+}$ representing the same element as $t_{i-1}(w)u_i(w)t_{i}(w)$. For this, $w'$ is derived from $w$ and with the canonical split sequence $\mathfrak k = (k_0, \ldots, k_{n-1})$, where 
\[
k_j =
\begin{cases}
    j &\text{if $j< i$,}\\
    j+1 & \text{if $j\geq i$.}
\end{cases}
\]
Note that any word $w'$ derived from $w$ can be obtained by successively applying the above elementary derivations.

Furthermore, if we have $0\leq i \leq j \leq \pabs{w}$ and a word $u'$ in $S_{t_i^-(w)}$ which satisfies 
\[
u' =_G t_{i-1}(w)u_i(w)t_i(w)\ldots u_j(w)t_j(w),
\]
then we define the \emph{$(w, (i, j))$-derived length of $u'$}, as
\[
\norm{u'}_{w, (i+1, j)} =1 + \sum_{k=i}^j (\abs{u_k(w)}+1).
\]
Again, when it is clear what $i$ and $j$ are, then we omit $(i, j)$ from the notation and call it the $w$--derived length and denote it by $\norm{u'}_w$.

We now list some properties of derivations.

\begin{lem}[Properties of derivations]\label{lemma:basic-properties-derivations}
Using the notation from Definition~\ref{def:derivation}, the following properties hold. \begin{enumerate}
    \item (same element) $w'  = _G w$
    \item (length invariance) $\sum_{i=1}^{\pabs{w'}} \norm{u_i(w')}_{w} = \abs{w}-1$.\label{prop:length-invariance}
    \item (transitivity) If $w''$ is derived from $w'$, then $w''$ is derived from $w$.
    \item (unchanged pieces) If $k_i -1 = k_{i-1}$, then $u_i(w') =_G u_{k_i}(w), t_{i-1}(w') = t_{k_{i-1}}(w)$, and $t_i(w') = t_{k_i}(w)$. We call such a piece $u_i(w')$ \emph{unchanged}.
    \item (trivial derivations) If $w$ is shortened, then every piece $u_i(w')$ is unchanged. In particular, $w'$ is shortened. We call a derivation where every piece is unchanged a \emph{trivial derivation}.
\end{enumerate}    
\end{lem}
\begin{proof}
    The lemma follows directly from the definition of a derivation.
\end{proof}

Since we are interested in efficiently computing derivations from a starting word $w$, it becomes important to distinguish derivations which are ``short'' in the sense that their length is bounded by the length of our initial word $w$. A precise definition of when we consider a derivation to be ``short-enough'' is given below.

\begin{defn}\label{def:non-sprawling}
    Let $K\geq 1$ be a constant and let $w'$ be a word derived from an organized word $w$. We say that $w'$ is $(K, w)$--non-sprawling if for all $1\leq i \leq \pabs{w'}$,
    \begin{align}\label{ineq:non-spraw}
        \abs{u_i(w')}\leq K\norm{u_i(w')}_{w},
    \end{align}
    for all pieces $u_i(w')$ of $w'$. Moreover, if $a = t_0(w')u_1(w') \ldots u_k(w')$ or $a =  t_0(w')u_1(w') \ldots u_k(w') t_k(w')$ is a prefix of $w'$, we say that $a$ is $(K, w)$--non-sprawling if \eqref{ineq:non-spraw} holds for all $1\leq i \leq k$.
\end{defn}

\begin{rem}\label{rem:quasi-implies-non-sprawling}
   Using the notation of Definition~\ref{def:non-sprawling}, if $u_i(w')$ is a $K$--quasi-geodesic in $G$ for all $1\leq i \leq \pabs{w'}$ (respectively for all $1\leq i \leq k$), then $w'$ (respectively $a$) is $(K, w)$--non-sprawling. Namely, by definition, $\norm{u_i(w')}_w - 1$ is the length of a word in $S$ representing the same element as $u_i(w')$. Hence, if $u_i(w')$ is a $K$--quasi-geodesic in $G$, then $\abs{u_i(w')}\leq K d + K \leq K\norm{u_i(w')}_w$, where $d$ is the length of a shortest word in $G$ representing the same element as $u_i(w')$.
\end{rem}

Being able to compute derivations which are non-sprawling will become important in the runtime analysis for our algorithms.

\subsection{Middle Derivations}\label{sec:middle-derivations}
Let $w$ be derived from the organized word $w^0$. We will describe a procedure which takes as input the pair $(w, w^0)$ and computes a word $w'\prec w$ which we will call a \emph{middle derivation} of the pair $(w, w^0)$. This is the crucial construction for our algorithms. As the definition of a middle derivation is somewhat technical, we first want to highlight the broad idea of what a middle derivation should be, see also Figure~\ref{fig:middle_deriv} for an illustration.


    \subsubsection*{Definition sketch of middle derivations.}
    \begin{figure}[h!]
    \centering
        \includegraphics[scale=0.9]{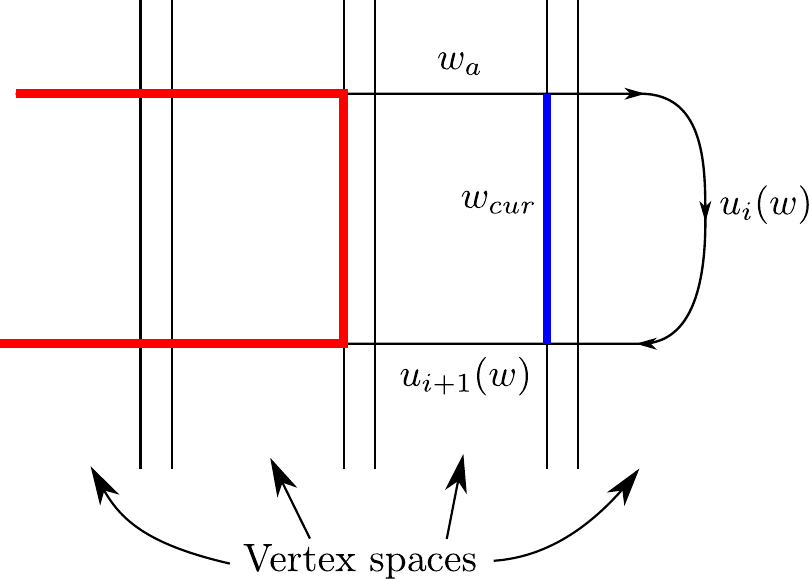}
        \caption{A middle derivation. Transitioning between any two given vertex spaces requires crossing an edge labelled by some $t\in E(\Gamma)$.}\label{fig:middle_deriv}
    \end{figure}
    
    The word $w$ gives a path $\theta$ in the Cayley graph of $\pi_1(\mathcal G)$ (with respect to our chosen generating set), which we think of as consisting of cosets of the vertex groups (henceforth, vertex spaces), and an edge labelled by some $t\in E(\Gamma)$ connects distinct vertex spaces. A piece of $w$ being reducible can be thought of as $\theta$ having  an ``excursion'' in a vertex space that starts and ends at the same coset of an edge group. We can perform a reduction replacing a subpath of $\theta$ with a path in a coset of an edge group contained in an adjacent vertex space; this is the content of Step \eqref{step:wcur} and Step \eqref{step:inner-reducibility-check} below, and the replacing path is labelled by $w_{cur}$ in the figure. This is the ``basic move'' of a middle derivation.

    In general, we check from left to right whether the current piece $u_i(w)$ is reducible, that is, we check whether we can perform a ``basic move'' as described above (this is done in Steps~\ref{step:1}, \ref{step:reducibility-check} and \ref{step:wcur}). If such a reduction is possible, instead of moving to the next piece $u_{i+1}(w)$, we check whether a further reduction around $u_i(w)$ is possible (that is, we check whether $u_{i-1}(w) w_{cur} u_{i+1}(w)$ is reducible\footnote{Since the left part of $w$ might have already been changed by the procedure, this is not quite correct. Actually, we instead have to check whether $w_aw_{cur} u_{i+1}(w)$ is reducible, where $w_a$ is the last piece of the possibly changed left part of $w$.}), see also Figure~\ref{fig:middle_deriv}. We repeat this until no more reduction around $u_i(w)$ is possible, \textbf{or} the sum of the (derived) lengths of the other paths involved exceeds the (derived) length of $u_i(w)$. This is what happens in Step~\ref{step:innerrep}. In other words, we continue checking whether we have a reduction that includes $u_i(w)$ until either there is no more of them, or the to be checked word has (derived) length at least twice that of $u_i(w)$. 

    While this might seem like an arbitrary cutoff, this is the crucial ingredient that will help us to show that not only can we compute middle derivations effectively (see Lemma~\ref{lemma:quick-middle-der-of-free-prod} and Lemma~\ref{lemma:compute_center_der_quickly}), but also, after logarithmically many successive middle derivations, there are no more pieces that are reducible (see Proposition~\ref{prop:log-derivation-is-shortenend}). 
    

\subsubsection*{Definition of a middle derivation.}
We now are ready to start the formal definition of a middle derivation.

\textbf{Notation:} Let $w$ be derived from an organized word $w^0$. Let $n = \pabs{w}$ be the path-length of $w$. For $1\leq i \leq n+1$ define $w_i = \Pi_{j=i}^n u_j(w)t_j(w)$. In other words, $w_i$ is a suffix of $w$ starting from the piece $u_i(w)$ and $w_{n+1}$ is the empty word. 

In the procedure we will start with defining $a = \alpha_\ast$ and $i=0$ and then update $a$ and $i$ such that (at most times) $a$ is semi-organized, $aw_i$ (or $aw_{i+1}$) is organized and $aw_i\prec w$ (or $aw_{i+1}\prec w$). We denote $t_{last}(a)$ and $u_{last}(a)$ by $t_a$ and $w_a$ respectively.

Throughout the procedure, we will add indices (or pairs of them) to the initially empty sets $I_{start}$, $I_{core}$, and $I_{red}$. These sets are only used for analysis purposes and do not affect the procedure. Further, we will have a variable $c$ which is a counter that affects how many times we repeat Step \ref{step:innerrep}. Note that the only thing we need $w^0$ for is to determine suitable values for $c$. Lastly, $w_{cur}$ will be a variable denoting a word which will usually be in an edge group.

\begin{enumerate}[label = \Alph*)]
    \item \textbf{Define $a := \alpha_\ast$ and $i:= 0$. }\gar{$aw_{i+1}  \prec w$,  $a$ is full, and $t_a = t_{i}(w)$.} \label{step:setup}
    \item \textbf{Repeat the following.} We call each repetition an \emph{AlgStep}. \gar{$aw_{i+1}  \prec w$,  $a$ is full, and $t_a = t_{i}(w)$.}\label{step:algrepeat}
\begin{enumerate}
    \item \textbf{Increase $i$ by one.} \gar{$aw_{new(i)}  \prec w$, $a$ is full, and $t_a = t_{new(i)-1}(w)$. } \label{step:1}
    \item  \textbf{Check whether $i = n+1$.}\label{step:endcheck} 
    \begin{itemize}
        \item \textbf{If $i = n+1$: Go to Step \ref{step:defofw'}} \gar{$aw_i = a  \prec w$.}
        \item \textbf{If $i\neq n+1:$} Define $i_\mathrm{start}$ as $i$ and add it to $I_{start}$. \textit{We call $i_{start}$ the instigator of the current repetition of AlgStep.}
    \end{itemize}
    \item \textbf{Check whether $u_i(w)$ is $w$--reducible.} Since $a$ is full and $t_a = t_{i-1}(w)$ (see the guarantee in Step~\ref{step:1}) we have that $u_i(w)$ is a piece of $aw_i$ and it is $aw_i$--reducible if and only if it is $w$--reducible. \label{step:reducibility-check}
        \begin{itemize}
            \item \textbf{ If $u_i(w)$ is $w$--irreducible: append $u_i(w)t_i(w)$ to $a$ and go to Step \ref{step:algrepeat}}. \gar{$new(a)w_{i+1}  \prec w$, $new(a)$ is full, and $t_{new(a)} = t_i(w)$.}
        \end{itemize}
    \item \textbf{Set $c$ to $\norm{u_i(w)}_{w^0}$} and add $i_{start}$ to $I_{core}$.\label{step:core-add}
    \item \textbf{Set $w_{cur}$ to a word over $S_{\bar{t}_a}$ which satisfies $w_{cur} = _G t_a u_i(w) t_a^{-1}$ and remove $t_a$ from $a$.} This can be done since $u_i(w)$ is  $old(a)w_i$--reducible. \gar{$new(a)new(w_{cur}) w_{i+1}\prec old(a) w_i\prec w$, $new(a)$ is not full.}\label{step:wcur}
    \item \textbf{Repeat the following}.  We call each repetition an InnerStep. \gar{$aw_{cur}w_{i+1}  \prec w$, $a$ is not full.}\label{step:innerrep}
    \begin{enumerate}[label = \roman*)]
        \item \textbf{Increase $i$.} \gar{$aw_{cur}w_{new(i)}\prec w$, $a$ is not full.} \label{step:inner-increase} 
        \item \textbf{Subtract $\norm{u_i(w)}_{w^0} + \norm{w_a}_{w^0}$ from $c$.}   \label{step:update_counter2}
        \item \textbf{Check whether $c\geq 0$.} \label{step:check-counter}
        \begin{itemize}
            \item \textbf{If $c < 0$: Append $w_{cur}u_i(w)t_i(w)$ to $a$. Go to step \ref{step:algrepeat}.} Add $(i_{start}, i)$ to $I_{red}$. \gar{$old(a)w_{cur}w_{i} = new(a)w_{i+1}  \prec w$, $a$ is full, and $t_{new(a)} = t_i(w)$.} 
        \end{itemize}
        \item \textbf{Check whether the piece $w_aw_{cur}u_i(w)$ is $aw_{cur}w_i$--reducible.}\label{step:inner-reducibility-check}
        \begin{itemize}
            \item \textbf{If it is irreducible: append $w_{cur}u_i(w)t_i(w)$ to $a$. Go to step \ref{step:algrepeat}.}\gar{$new(a)w_{i+1}\prec w$, $a$ is full, and $t_a = t_i(w)$.} 
            \item \textbf{If it is reducible: replace $w_{cur}$ with a word over $S_{\bar{t}_a}$ such that $new(w_{cur}) = _G t_aw_a old(w_{cur})u_i(w)t_{i}(w)$. Remove $t_aw_a$ from $a$.} \gar{$new(a)new(w_{cur})w_{i+1} \prec w$, $new(a)$ is not full.}
        \end{itemize}
    \end{enumerate}

\end{enumerate}
\item \textbf{Define $w'$ as $a$.} \gar{$w'  \prec w$.} \label{step:defofw'}
\end{enumerate}

\begin{lem}
    The procedure is well defined and produces a word $w'$ which is a derivation from $w$.
\end{lem}
\begin{proof}
    The guarantee of Step \ref{step:setup} trivially holds. It suffices to show (by induction) that if all guarantees hold up to a certain step, then the next guarantee after this step also hold. This follows directly from each step and the choice of guarantees.
\end{proof}

\begin{prop}[Properties of middle derivations]\label{proposition:properties-of-middle-derivations}
The following properties hold.
\begin{enumerate}
    \item If $i\in I_{core}$, then $u_i(w)$ is $w$--reducible.
    \item If $u_j(w')$ is $w$--reducible and $i\in I_{start}$ is the largest index such that $u_j(w')$ was changed in the $i$-AlgStep, then $(i, k)\in I_{red}$ for some $k$. Furthermore $\norm{u_j(w')}_{w^0}\geq 2\norm{u_i(w)}_{w^0}$.\label{prop:doubling}
   \item The word $w$ is shortened if and only if we never reach Step \ref{step:core-add}, in which case $w'$ is a trivial derivation of $w$ (and hence shortened).
\end{enumerate}
\end{prop}

\begin{proof}
    Property 1 follows directly from Step \ref{step:reducibility-check}. Property 2 follows from the fact that pieces in $a$ are only changed/added in Steps \ref{step:reducibility-check}, \ref{step:inner-reducibility-check} and \ref{step:check-counter}. In Steps \ref{step:reducibility-check} and \ref{step:inner-reducibility-check} all we do is add an irreducible piece to a full semi-word, so if there is a reducible subword in the end, it cannot come from there. Thus, if we have a $w'$--reducible piece $u_\ell(w')$ it needed to come from $c<0$ in Step \ref{step:check-counter}. Let $(i, k)\in I_{red}$ be the corresponding pair of integers we added when $c<0$. Having $c<0$ means that 
    \begin{align*}
        \norm{u_i(w)}_{w^0} < \sum_{j = i+1}^k \norm{u_j(w)}_{w^0}.
    \end{align*}
    Furthermore, $u_i(w), u_{i+1}(w), \ldots u_{k}(w)$ are all in the $w$--base of $u_\ell(w')$. In particular, 
    \[
        \norm{u_\ell(w')}_{w^0}\geq \sum_{j = i}^k \norm{u_j(w)}_{w^0} \geq 2\norm{u_i(w)}_{w^0}.
    \]

    Property 3 follows from the following observation: if we reach Step \ref{step:core-add}, then $w$ is not shortened since $u_i(w)$ is $w$--reducible. If we never reach Step \ref{step:core-add}, then none of the pieces of $w$ are $w$--reducible, and hence $w$ is shortened.
\end{proof}

Property \ref{prop:doubling} is the key in showing that after $\log(\abs{w^0})$ many middle derivations we get an organized word which is shortened. This is the content of the following proposition.

\begin{prop}\label{prop:log-derivation-is-shortenend}
    Let $w^0$ be an organized word and let $m = \abs{w^0}$. For $i\geq 1$, define $w^i$ as a middle derivation of the pair $(w^{i-1}, w^0)$. Then $w^{\lceil \log (m) \rceil}$ is shortened.
\end{prop}

\begin{proof}
    We inductively show the following statement. If a piece $u_j(w^i)$ is $w^i$--reducible, then $\norm{u_j(w^i)}_{w^0}\geq 2^i$.

    Length invariance (Lemma \ref{lemma:basic-properties-derivations} \ref{prop:length-invariance}) implies that for all $i$, no piece of $w^i$ can have derived length more than $m$. Thus, the induction hypothesis shows that $w^{\lceil \log (m) \rceil}$ is indeed shortened.

    For $i = 0$, the induction hypothesis is trivial, since every piece has derived length at least $1$. Next assume that the induction hypothesis holds for $i$. We will show that it holds for $i+1$. Let $u_j(w^{i+1})$ be a $w^{i+1}$--reducible piece. Proposition \ref{proposition:properties-of-middle-derivations}\eqref{prop:doubling} shows that $\norm{u_j(w^{i+1})}_{w^0}\geq 2 \norm{u_k(w^i)}_{w^0}$ for some $w^i$--reducible piece $u_k(w^i)$. The induction hypothesis for $i$ implies that $\norm{u_k(w^i)}_{w^0}\geq 2^i$ and hence $\norm{u_j(w^{i+1})}_{w^0}\geq 2^{i+1}$, which shows that the induction hypothesis for $i+1$ indeed holds.
\end{proof}

\subsection{Storing organized words on a Turing machine}\label{sec:storing-normal-words}

We now explain how to efficiently store the information needed to perform middle derivations  on a Turing machine.

If we have a word $w$ which is organized and derived from a word $w_0$, we will save this on a Turing machine on three bands 
$\word{b}, \edge{b}, \leng{b}$ as follows. 
\begin{itemize}
    \item $\word{b}$ reads the word $\ast u_1(w) \# u_2(w) \# u_3(w) \# \ldots \# u_{n-1}(w)\# u_n(w)\ast$.
    \item $\edge{b}$ reads the word $\ast t_0(w)t_1(w)\ldots t_n(w)\ast$.
    \item $\leng{b}$ reads the word $\ast \ell_1\#\ell_2\#\ldots \#\ell_{n-1}\#\ell_n\ast$, where $\ell_i$ are words in the alphabet $\{1\}$ and the length of $\ell_i$ is equal to the derived length $\norm{u_i(w)}_{w^0}$ for all $1\leq i \leq n$.
\end{itemize}
Here $n = \pabs{w}$ denotes the path-length of $w$. We denote a set of bands $(\word{b}, \edge{b}, \leng{b})$ by $b_*$ and say that $(w,w^0)$ is stored on $b_*$ organizedly.

\begin{lem}\label{lem:compute-normal-form}
     Let $w'$ be a word in $S$. Then in time $\mc O(\abs{w'})$ we can compute a word $w$ which is organized, satisfies $w =_G w'$ and has $\abs{w}\leq \abs{V(\Gamma)}\abs{w'}+2$. Furthermore, if we have a word $w$ which is organized, and bands $b_*$ (which are empty) then we can store $(w, w)$ organizedly on $b_*$ in time $\mc O(\abs{w'})$.
\end{lem}
\begin{proof}
    Recall that since $\mc G$ is a graph of groups, there is a spanning tree $T$ of its underlying graph $\Gamma$ such that $t =_G 1$ for all edges $t\in T$. For each pair of vertices $u, v\in \Gamma$, we precompute a path $(t_1, \ldots, t_k)$ which is a geodesic in $T$ from $u$ to $v$. Note that $k< \abs{V(\Gamma)}$ for any choice of $u, v$.

    We go through letters $c$ in $w'$ from left to right, while remembering a vertex $v\in V(\Gamma)$ which should be thought of as the vertex group ``we are currently in''. Initially, we define $w : = \alpha_*$ and we have that $v = v_0$. When looking at a letter $c$ of $w'$ we do the following. If $c= \alpha_*$, we do not do anything. If $c \in E(\Gamma)$, then we add $t_1\ldots t_kc$ to $w$, where $(t_1, \ldots , t_k)$ is the geodesic in $T$ from $v$ to $c^-$. We then set $v$ to $c^+$. On the other hand, if $c\in S_u$ for some vertex $u\in V(\Gamma)$, then we add $t_1\ldots t_kc$ to $w$ where $(t_1, \ldots , t_k)$ is the geodesic from $v$ to $u$ in $T$. We then set $v$ to $u$. At the end, we add $\alpha_*$ to $w$.

    For each letter of $w'$, we added at most $\abs{V(\Gamma)}$ letters to $w$ (plus the two letters $\alpha_*$ at the very end and the very beginning) hence $\abs{w}\leq \abs{V(\Gamma)}\abs{w'} +2$.

    The second part is straightforward. We first add $\ast$ to all the bands in $b_*$, and add $\alpha_*$ to band $\edge{b}$. Then, starting from the second letter in $w$, we go through the letters in the word $w$ from left to right. If the current letter $c$ is in $E(\Gamma)$, then we add $\#$ to $\word{b}$, $1\#$ to $\leng{b}$ and $c$ to $\edge{b}$. If $c = \alpha_*$, then we add $\ast$ to $\word{b}$, $1\ast$ to $\leng{b}$ and $c$ to $\edge{b}$. Otherwise, we add $c$ to $\word{b}$ and $1$ to $\leng{b}$. It is easy to verify that this procedure leads to the desired result. 
\end{proof}

In view of what we have done so far, we now point out that to solve efficiently the word problem for $\pi_1(\mc{G})$ we only need to know that the word problem in the vertex groups can be solved efficiently, and that we can efficiently compute middle derivations.

\begin{cor}
\label{cor:A12_implies_wp}
    Let $\mc G$ be a graph of groups and $f(k)\geq k$ be a convex function and $K\geq 1$ an integer. Suppose the following hold.
    \begin{enumerate}[label = (A\arabic*)]
        \item We have an algorithm which decides, for each vertex group $G_v$, whether a given word $w\in G_v$ represents the identity, and does so in time at most $f(\abs{w})$.\label{alg:vertex-group-trivial}
        \item We have an algorithm which computes, given an organized word $w^0$ and a $(K, w^0)$--non-sprawling derivation $w$ of $w^0$ which is stored organizedly, a $(K, w^0)$--non-sprawling middle derivation of $(w, w^0)$ (also stored organizedly), outputs whether the middle derivation was the trivial derivation, and does so in time at most $f(\abs{w^0})$.\label{alg:middle-der}
    \end{enumerate}
    Then the word problem in $G = \pi_1(\mc G)$ can be solved in time at most $\log(m)f(m)$ time, where $m$ denotes the length of the input word.
\end{cor}
\begin{proof}
    Let $w$ be a word over $S$. Lemma~\ref{lem:compute-normal-form} states that we can compute an organized word $w^0$ in time at most $\abs{w}$ and $\abs{w^0}\preceq \abs{w}$. For $i\geq 1$ we can now use Algorithm~\ref{alg:middle-der} to inductively compute a $(K, w^0)$--non-sprawling middle derivation $w^i$ of $(w^{i-1}, w^0)$. If $\pabs{w^i} = 1$, we can use Algorithm~\ref{alg:vertex-group-trivial} to check whether $u_1(w^i)$ is trivial. If so, we know that $w$ represents the trivial word and we stop. If not we know that $w$ does not represent the trivial word and we stop. If instead $\pabs{w^i}\geq 2$ and if Algorithm~\ref{alg:middle-der} outputted that $w^i\prec w^{i-1}$ was a trivial derivation, then we know by Lemma~\ref{proposition:properties-of-middle-derivations} that $w^i$ (and hence $w$) does not represent the trivial word and we stop. If none of the above happens, we use Algorithm~\ref{alg:middle-der} to compute a middle derivation $w^{i+1}\prec w^i$ and repeat the above step. Proposition~\ref{prop:log-derivation-is-shortenend} shows that the procedure stops at the latest when $i = \lceil \log (\abs{w^0}) \rceil$. Every step can be done in time at most $f(\abs{w^0})$ and hence the algorithm takes at most $\lceil \log (\abs{w^0}) \rceil f(\abs{w^0})$ time.
\end{proof}

\subsection{Middle derivations of free products}\label{sec:prop1.3proof}

In this section, we will show how to compute the middle derivation efficiently in the case where all edge groups are trivial. In particular, this includes free products. In combination with Corollary~\ref{cor:A12_implies_wp}, this will yield a proof of Proposition~\ref{prop:free_product}.

\freeprod*

The main tool in the proof of Proposition~\ref{prop:free_product} is the following Lemma, which states that we can compute the middle derivation efficiently. 

\begin{lem}\label{lemma:quick-middle-der-of-free-prod}
Let $\mc G$ be a graph of groups with trivial edge groups and let $f(k)\geq k$ be a convex function. Suppose the following holds.
\begin{enumerate}[label = (A1)]
    \item We have an algorithm, which determines, given a vertex $v\in V(\Gamma)$ and a word $w$ over $S_{v}$ whether $w$ represents the trivial element in $G_v$ and does so in time at most $f(\abs{w})$. \label{alg:free-product-edge}
\end{enumerate}
Then, given a $(1, w^0)$-non-sprawling derivation $w$ of $w_0$ stored organizedly, we can compute a $(1, w^0)$-non-sprawling middle derivation $w'$ (stored organizedly) of $(w, w_0)$ in time at most $f(\abs{w_0})$.
\end{lem}

In this case, the key is that $w_{cur}$ is trivial at all times, which allows us to compute a middle derivation efficiently. To prove Lemma \ref{lemma:quick-middle-der-of-free-prod} it is enough to show that every iteration of \step can be done in time at most $f(\norm{u_i(w)}_{w^0}) + \norm{u_i(w)}_{w^0} + \norm{u_j(w)}_{w^0}$, where $i$ is the instigator of that iteration and $j+1$ is the instigator of the next iteration. To show this, we annotate to each step in the definition of a middle derivation how much time it takes to perform this step (and if unclear, how to perform the step).

\textbf{Detailed setup.} We have bands $\mf w_*$ and $\mf a _*$ where we store the word $(w, w^0)$ and $(a, w^0)$ in (semi)-organizedly. We have bands $\mf c_{counter}$ and $\mf c_{leng}$ where we store the value of $c$ and the $w^0$--derived length of $w_{cur}$. Further, we have auxiliary bands where we run the algorithm \ref{alg:free-product-edge}.

At the beginning of every step, we aim to ensure the following: 
\begin{itemize}
    \item the head of any of band $\mf w_*$ is at the end of its $i$-th segment.
    \item the head of any band of $\mf a_*$ is at the end of its band. 
\end{itemize}
Recall that since $w$ is $(1, w^0)$-non-sprawling we have that $\abs{u_i(w)}\leq \norm{u_i(w)}_{w^0}$ for all $1\leq i \leq n$. We ensure that at all points in time that $a$ is $(1, w^0)$-non-sprawling.

With this, it takes at most $\norm{u_{new(i)}(w)}_{w^0}$ time to increase $i$, since we also have to move the heads of the bands $\mf w_*$.

At the beginning of every \step we aim to ensure that: 
\begin{itemize}
    \item the bands $\mf c_{counter}$ and $\mf c_{leng}$ are empty.
\end{itemize}

We do not store the value of $i$ (we ``know'' the value of $i$ by the position of the heads of the $\mf w_*$ bands). Similarly, we do not store $i_{start}$ or the sets $I_{start}, I_{cen}$. Those are purely used for the analysis.

Now, in the case of free products, we turn the middle derivation procedure into an algorithm, by embellishing it and explaining how things are stored on the Turing machine $\mc M$. 

We use the same notation as in the definition of a middle derivation. In particular,  $n = \pabs{w}$ denotes the path-length of $w$. For $1\leq i \leq n+1$, $w_i = \Pi_{j=i}^n u_i(w)t_i(w)$. Further $t_{last}(a)$ and $w_{last}(a)$ are denoted by $t_a$ and $w_a$ respectively.

\begin{enumerate}[label = \Alph*)]
    \item \textbf{Define $a := \alpha_\ast$ and $i:= 0$. } This can be done in constant time. \label{step:setup''}
    \item \textbf{Repeat the following.}\label{step:algrepeat''}
\begin{enumerate}
    \item \textbf{Increase $i$ by one.} \tim{This takes at most $\norm{u_{new(i)}}_{w^0}$ time.} \label{step:1''}
    \item  \textbf{Check whether $i = n+1$.} This can be done by checking whether the cell after the current position of the head of $\edge{\mf w}$ reads the letter $\ast$. \tim{This takes constant time.}\label{step:endcheck''} 
    \begin{itemize}
        \item \textbf{If $i = n+1$: Go to Step \ref{step:defofw'''}} \tim{This takes constant time.}
        \item \textbf{If $i\neq n+1:$} Do nothing. (Since the set $I_{start}$ is for analysis purposes only, the actual algorithm does not have to do anything) \tim{This takes constant time.}
    \end{itemize}
    \item \textbf{Check whether $u_i(w)$ is $w$--reducible.} We first have to check whether $t_{i}(w) = \bar{t}_a$ and if so, we have to use the algorithm \ref{alg:free-product-edge} to check whether $u_i(w)$ represents the trivial element. \tim{This takes at most $f(\abs{u_i(w)})\preceq f(\norm{u_i(w)}_{w^0})$ time.} \label{step:reducibility-check''}
        \begin{itemize}
            \item \textbf{ If $u_i(w)$ is $w$--irreducible:} append $u_i(w)t_i(w)$ to $a$ and go to Step \ref{step:algrepeat}. \tim{This takes at most $\norm{u_i(w)}_{w^0}$ time.} Namely, we have to append $ t_i(w)$ to $\edge{\mf a}$, append $\# 1^{\norm{u_i(w)}_{w^0}}$ to $\leng{\mf a}$ and append $\# u_i(w)$ to $\word{\mf a}$. \gar{$new(a)$ is $(1, w^0)$--non-sprawling.}
        \end{itemize}
    \item \textbf{Set $c$ to $\norm{u_i(w)}_{w^0}$.} We have to update $\mf c_{counter}$. \tim{This takes at most $\norm{u_i(w)}_{w^0}$ time.} \label{step:core-add''}

    \item \textbf{Set $w_{cur}$ to a word over $S_{\bar{t}_a}$ which satisfies $w_{cur} = _G t_a u_i(w) t_a^{-1}$ and remove $t_a$ from $a$.} Since all edge groups are trivial, we can choose $w_{cur}$ as the trivial word and hence do not have to store it anywhere. However, we have set $c_{leng}$ to $\norm{u_i(w)}_{w^0}$. \tim{This takes at most $\norm{u_{i_{start}}(w)}_{w^0}$ time.} \gar{$ c = 2\norm{u_{i_{start}}(w)}_{w^0} - \norm{w_{cur}}_{w^0}$.} \label{step:define_w_cur''}

    \item \textbf{Repeat the following}.\label{step:innerrep''}
    \begin{enumerate}[label = \roman*)]
        \item \textbf{Increase $i$.}  \tim{This takes at most $\norm{u_{new(i)}(w)}_{w^0}$ time.}\label{step:inner-increase''} 
        \item \textbf{Subtract $\norm{u_i(w)}_{w^0} + \norm{w_a}_{w^0}$ from $c$.}  If $c$ is negative, we do not care about the precise value of it, so once $c$ becomes negative, we can stop subtracting. Define $c_{dif} =\min\{ \norm{u_i(w)}_{w^0} + \norm{w_a}_{w^0}, old(c)+1\}$, note that $c_{dif}$ is the exact amount we subtract from $c$. \tim{This step takes at most $c_{dif}$ time. }\gar{ $c = 2\norm{u_{i_{start}}(w)}_{w^0} - \norm{w_{cur}}_{w^0} - c_{dif}$}  \label{step:update_counter2''}
        \item \textbf{Check whether $c\geq 0$.} \tim{This takes constant time}\label{step:check-counter''}
        \begin{itemize}
            \item \textbf{If $c < 0$: Append $w_{cur}u_i(w)t_i(w)$ to $a$. Go to step \ref{step:algrepeat}.} Again, we have to update all bands in $\mf a_{\ast}$. Namely, we have to append $t_i(w)$ to $\edge{\mf a}$, append $\# 1^{\norm{w_{cur}}_{w^0}+\norm{u_i(w)}_{w^0}}$ to $\leng{\mf a}$ and append $ u_i(w)$ to $\word{\mf a}$. \tim{This takes at most $\norm{u_i(w)}_{w^0} + \norm{u_{i_{start}}(w)}_{w^0}$ time.} Lastly, we have to reset $c_{leng}$ and $c_{counter}$. \tim{This takes at most $ \norm{u_{i_{start}}(w)}_{w^0}$ time.} Note that $\norm{w_a}_{w^0}$ increases by $\norm{u_i(w)}_{w^0} + \norm{w_{cur}}$ while $\abs{w_a}$ increases by $\abs{u_i(w)}$. \gar{$new(a)$ is $(1, w^0)$--non-sprawling.}
        \end{itemize}

        \item \textbf{Check whether the piece $w_aw_{cur}u_i(w)$ is $aw_{cur}w_i$--reducible.} We first check whether $t_a = \bar{t}_i(w)$ and if so, we use Algorithm~\ref{alg:free-product-edge} to check whether $w_au_i(w)$ is trivial.\tim{This takes at most $ f(\abs{w_a} + \abs{u_i(w)})\preceq f(c_{dif})$--time}\label{step:inner-reducibility-check''}
        \begin{itemize}
            \item \textbf{If it is irreducible: append $w_{cur}u_i(w)t_i(w)$ to $a$. Go to step \ref{step:algrepeat}.} Again, we have to update all bands in $\mf a_{\ast}$ as if $c<0$ and reset the bands $c_{leng}$ and $c_{counter}$. \tim{This takes at most $ \norm{u_i(w)}_{w^0} + \norm{u_{i_{start}}(w)}_{w^0}$ time.} \gar{$new(a)$ is $(1, w^0)$--non-sprawling.}
            \item \textbf{If it is reducible: replace $w_{cur}$ with a word over $S_{\bar{t}_a}$ such that $new(w_{cur}) = _G t_aw_a old(w_{cur})u_i(w)t_{i}(w)$. Remove $t_aw_a$ from $a$.} We again choose $w_{cur}$ as the trivial word and hence do not have to store it anywhere. We have to update $c_{leng}$ and the bands $\mf{a}_\ast$. \tim{This takes at most $ c_{dif}$ time.} \gar{ $c = 2\norm{u_{i_{start}}(w)}_{w^0} - \norm{w_{cur}}_{w^0} - c_{dif}$}
        \end{itemize}
    \end{enumerate}

\end{enumerate}
\item \textbf{Define $w'$ as $a$.} \gar{$w'$ is $(1, w^0)$--non-sprawling.}\label{step:defofw'''}
\end{enumerate}

We want to bound the runtime of every AlgStep. Let $i$ be its instigator and let $j$ be the instigator of the next step. Steps~\ref{step:1''} to \ref{step:define_w_cur''} take at most $f(\norm{u_i(w)}_{w^0})$ time. We claim that Step~\ref{step:innerrep''} takes at most $ f(\norm{u_i(w)}_{w^0}) + \norm{u_{j-1}(w)}_{w^0}$--time. Namely, any iteration of InnerStep, unless it is the last iteration of InnerStep, takes at most $f(c_{dif})$ time. Since $c_{dif}$ gets subtracted from $c$ (whose initial value is $\norm{u_i(w)}_{w^0}$) and $c$ does not become negative unless it is the last iteration of InnerStep, the total time to perform all but the last iteration of \sstep is at most $f(\norm{u_i(w)}_{w^0})$, where we used that $f(k)\geq k$ is convex. The last iteration of \sstep takes at most $ \norm{u_{j-1}(w)}_{w^0} + f(\norm{u_{i}(w)}_{w^0})$ time.  

Using that $f(k)\geq k$ is a convex function and Lemma~\ref{lemma:basic-properties-derivations}\eqref{prop:length-invariance}, we get that the total runtime of this implementation of the middle derivation takes at most $f(\abs{w^0})$--time, which concludes the proof of Lemma~\ref{lemma:quick-middle-der-of-free-prod}.

\begin{proof}[Proof of Proposition~\ref{prop:free_product}]
    This follows directly from Lemma~\ref{lemma:quick-middle-der-of-free-prod} and Corollary~\ref{cor:A12_implies_wp}.
\end{proof}

\begin{rem}\label{rem:divide-and-conquer}
Saul Schleimer suggested to us a different but not dissimilar approach to solve the word problem for free products which has a nice description in terms of a recursion. The recursive nature makes it intricate to implement on a Turing machine, and we preferred here to present our method in order to be as explicit and self-contained as possible regarding the implementation. Namely, one could also use a ``divide-and-conquer'' approach which we now describe. Recursively, starting with a word $a$ of length $n$, we would like to split it into two halves $b$ and $c$, and first solve the word problem for $b$ and $c$. More precisely, we would like to put $b$ and $c$ in shortened form, so that  a word $a =  bc$, where we know that $b$ and $c$ are shortened and now we have to ensure that $bc$ is shortened. To ensure this we only have to figure out whether the piece at the intersection of $b$ and $c$ is reducible, and if so continue reducing until we cannot reduce anymore (for example because we reached the other end of $b$ or $c$). Instead of doing this recursively, to implement it on a Turing machine one can use the ``bottom-up'' approach, that is, starting with $i=0$ and increasing $i$ in every round, we split $w$ into $2^{\log{m}-i}$ words of length $2^i$ and solve the word problem for those words. For $i = 0$ this is trivial and we can make sure that the subwords are shortened. For $i > 1$, we discussed above how to solve the word problem for $a=bc$ where $b$ and $c$ are shortened. With this, bringing $a$ into a shortened form takes $f(2^{i+1})$--time. Thus in total, solving the word problem takes $\log(m) f(m)$--time. We note that one needs to keep track on the Turing machine of where the segments to ``merge'' are at each step, and we believe that this is about as difficult as keeping track of the derived length in our method.
\end{rem}

\subsection{Storing words and elements on a Turing machine for admissible graph of groups}\label{sec:storing_in_edge_groups}

To compute the middle derivation, we have to store and manipulate words that lie in given edge or vertex groups efficiently. Luckily, the geometry of admissible graph of groups makes this easier because of various facts about the structure of edge and vertex groups that are summarised in Lemma \ref{lem:properties_of_vertex_groups}, which we recall below for convenience.

\edgegroups*

Let $\alpha\in E(\Gamma)$ be an edge. We say that a word $w$ is in \emph{edge normal form for the edge $\alpha$} if $w = c_\alpha^ko_{\alpha}^\ell h$ for some $k, \ell\in \Z$ and $h\in H_\alpha$. 

Thus for a given edge $\alpha\in E(\Gamma)$ and an element $g\in \tau_{\alpha}(G_{\alpha})$, there are unique integers $k, \ell$ and a unique coset representative $h\in H_{\alpha}$ such that $g = c_{\alpha}^{k}o_{\alpha}^{\ell}h$. We write $(g, \alpha)\equiv (k, \ell, h)$. We can store the pair $(g, \alpha)\equiv (k, \ell , h)$ on a Turing machine as follows. We can save $\alpha$ and $h$ in the state (since there are only finitely many options for those) and have two bands $b_1$ and $b_2$ where we store the integers $k$ and $\ell$ respectively and whose heads are at the end of their bands.  

Assume we have elements $(g, \alpha) \equiv (k, \ell, h)$ and $(g', \alpha) \equiv (k', \ell', h')$ stored. We can replace $(g, \alpha)$ by $(gg', \alpha)$ (or equivalently by $(g'g, \alpha)$) as follows. Add $k'+ k''$ to $k$, add $\ell' + \ell''$ to $\ell$ and replace $h$ by $h''$, where $(k'', \ell'', h'') \equiv (hh', \alpha)$. For this step we use that all edge groups in admissible graph of groups are abelian by definition. Since there are only finitely many cosets, $\abs{k''}$ and $\abs{\ell''}$ are bounded by a constant. Hence, replacing $(g, \alpha)$ by $(gg', \alpha)$ only takes $\mc O(\abs{k'}+\abs{\ell'}) = \mc O(\abs{g'})$ time.

Another operation we want to perform is the replacement of an element $(g, \alpha)\equiv (k, \ell, h)$ with the element $(t_{\alpha}g t_{\alpha}^{-1}, \bar{\alpha})$. Property \ref{prop:conjugate_by_edge} implies that $(t_{\alpha}g, t_{\alpha}^{-1}, \bar{\alpha})\equiv (\ell + k', k+ \ell', h')$, where $(k', \ell', h') \equiv (t_{\alpha}ht_{\alpha}^{-1}, \bar{\alpha})$. Again, since there are only finitely many cosets, $\abs{k'}$ and $\abs{\ell'}$ are bounded by a constant. However, the naive way to replace $(g, \alpha)$ with $(t_{\alpha}g, t_{\alpha}^{-1}, \bar{\alpha})$ still takes $\mc O (\abs{g})$ time, since we need to copy what is written in band $b_1$ to band $b_2$ and vice versa.

If we are a little bit more careful with how we store elements $(g, \alpha)\equiv (k, \ell, h)$, then such a replacement can be done in constant time. Namely, instead of saving $k$ on $b_1$ and $\ell$ on $b_2$, we will (in the state) remember indices $i_{cen}$ and $i_{ort}$ (which take distinct values in $\{1, 2\}$) and save $k$ on $b_{i_{cen}}$ and $\ell$ on band $b_{i_{ort}}$. With this, it is easy to replace $(g, \alpha)$ by $(t_{\alpha}g, t_{\alpha}^{-1}, \bar{\alpha})$; we only have to switch $i_{cen}$ and $i_{ort}$ and then add $k'$ and $\ell'$ to the bands $b_{i_{cen}}$ and $b_{i_{ort}}$ respectively.

To summarize: 
\begin{itemize}
    \item Replacing $(g, \alpha)$ by $(gg', \alpha)$ takes $\mc O(\abs{g'})$ time.
    \item Replacing $(g ,\alpha)$ by $(t_{\alpha}g, t_{\alpha}^{-1}, \bar{\alpha})$ takes $\mc O (1)$ time.
\end{itemize}

\subsection{Computing the middle derivation for admissible graph of groups efficiently}

The goal of this section is to prove the following Lemma, which states that we can compute the middle derivation efficiently in admissible graphs of groups.

\begin{lem}\label{lemma:compute_center_der_quickly}
Let $\mc G$ be an admissible graph of groups and let $K$ be the constant from Lemma~\ref{lem:properties_of_vertex_groups}.

Then, given a $(K, w^0)$-non-sprawling derivation $w$ of $w_0$ stored organizedly we can compute a $(K, w^0)$-non-sprawling middle derivation $w'$ of $(w, w^0)$ (stored organizedly) in time at most $ \abs{w^0}$.
\end{lem}

We first need a preliminary lemma, which states the existence of an algorithm analogous to the one required by Lemma \ref{lemma:quick-middle-der-of-free-prod} for the free product case. 

\begin{lem}\label{lemma:there-exists-alg-edge-group}
Let $\mc G$ be an admissible graph of groups. There exists an algorithm that runs in $\mc O(n)$ that does the following.
\begin{enumerate}[label = (A2)]
    \item Given an edge $\alpha\in E(\Gamma)$ and a word $w$ over $S_{\alpha^+}$ with $\abs{w}\leq n$, determine whether $w$ represents an element in $\tau_{\alpha}(G_{\alpha})$ and if so compute $k, l$ and $h$ such that $(w, \alpha) \equiv(k, \ell, h)$.\label{alg:in-egde-group-check}
\end{enumerate}
\end{lem}

\begin{proof}
    By Definition \ref{def:admissible}-\eqref{item:edge_groups}, an element $q$ of $G_{\alpha^+}$ lies in $\tau_\alpha(G_\alpha)$ if and only if $\pi_{\alpha^+}(g)$ lies in the subgroup $\langle a_\alpha\rangle$ of $F_{\alpha^+}$. This is a quasiconvex subgroup, so we can use the algorithm given by Proposition \ref{prop:quasi_convex} to check membership of $w$ in $\tau_\alpha(G_\alpha)$ in linear time. Moreover, choosing the generating set to which we apply the proposition to consist of the projections to $F_{\alpha^+}$ of $o_\alpha^{\pm1}$ and $H_\alpha$ (recall that $\pi_{\alpha^+}(c_\alpha)$ is trivial by Lemma \ref{lem:properties_of_vertex_groups}-\eqref{item:abelian}) we get as output of the algorithm a word $w'$ in $o_\alpha$ and $H_\alpha$ which represents an element of $G_{\alpha^+}$ with the same projection to $F_{\alpha^+}$ as the element represented by $w$. Moreover, the length of $w'$ is also in $\mc O(n)$. Scanning through $w'$ and using rewriting rules we can convert $w'$ into another word $w''$ in $o_\alpha^{\pm 1}$ and $H_\alpha$, again of length $\mc O(n)$ and with the same projection to $F_{\alpha^+}$, with exactly one letter $h$ from $H_\alpha$; this is indeed the required $h$ output, and the sum of the exponents of the occurrences of $o_\alpha$ is the required $\ell$. Finally, we have to determine $k$. If we append $w''$ to $w$ we obtain a word representing $c_\alpha^k$, so that $k$ can be determined in linear time using Proposition \ref{prop:central_word_problem}.
\end{proof}

We now show how to efficiently implement the middle derivation algorithm for admissible graph of groups, thereby proving Lemma~\ref{lemma:there-exists-alg-edge-group}. That is, we will explain how to turn the procedure from Section~\ref{sec:middle-derivations} into an efficient algorithm.

\textbf{Detailed Setup.} We have bands $\mf w_*$ and $\mf a _*$ where we store the word $w$ and $a$ (semi-)organizedly. We have bands $\mf c_e$ where we store a representative of the word $w_{cur}$. We have bands $\mf c_{counter}$ and $\mf c_{leng}$ where we store the value of $c$ and the $w^0$--derived length of $w_{cur}$. Further, we have auxiliary bands where we run the algorithm \ref{alg:in-egde-group-check}.

At the beginning of every step, we aim to ensure the following: 
\begin{itemize}
    \item the head of any of band $\mf w_*$ is at the end of its $i$-th segment.
    \item the head of any band of $\mf a_*$ is at the end of its band. 
\end{itemize}
Recall that since $w$ is $(K, w^0)$-non-sprawling we have $\abs{u_i(w)}\leq K\norm{u_i(w)}_{w^0}$ for all $1\leq i \leq n$. At all points in time we ensure that $a$ is $(K, w^0)$-non-sprawling.

With this, it takes $\mc O(\norm{u_{new(i)}(w)}_{w^0})$ time to increase $i$, since we also have to move the heads of the bands $\mf w_*$.

At the beginning of every \step we aim to ensure that: 
\begin{itemize}
    \item the bands $\mf c_e, \mf c_{counter}$ and $\mf c_{leng}$ are empty.
\end{itemize}

At the beginning of every \sstep we aim to ensure that:
\begin{itemize}
    \item the word $(g, \alpha)$ stored on the bands $\mf c_e$ represents the same element as $w_{cur}$
\end{itemize}

We do not store the value of $i$ (we ``know'' the value of $i$ by the position of the heads of the $\mf w_*$ bands). Similarly, we do not store $i_{start}$ or the sets $I_{start}$ and $I_{core}$, since they are not used during the algorithm. However, we use the those sets for analyzing the algorithm. 

Now we embellish the procedure which computes a middle derivation and turn it to an algorithm, showing how things are stored in the Turing machine $\mc M$. We will use the notation from the definition of the middle derivation.

\begin{enumerate}[label = \Alph*)]
    \item \textbf{Define $a := \alpha_*$.} \tim{This can be done in constant time.}\label{step:setup'}
    \item \textbf{Repeat the following.} We call each repetition an \emph{AlgStep}. \label{step:algrepeat'}
\begin{enumerate}
    \item \textbf{Increase $i$ by one.} \tim{This takes at most $\norm{u_{new(i)}(w)}_{w^0}$ time. }\label{step:1'}
    \item  \textbf{Check whether $i = n+1$.}\label{step:endcheck'} This can be done in constant time, we just have to check whether the cell after the current position of the head of $\edge{\mf w}$ reads the letter $\ast$. 
    \begin{itemize}
        \item \textbf{If $i = n+1$: Go to Step \ref{step:defofw''}} \tim{This takes constant time.}
        \item \textbf{If $i\neq n+1:$} Do nothing. \textit{Recall that for bookkeeping, we define $i_\mathrm{start}$ as $i$ and add it to $I_{start}$, but this is not represented on the Turing machine.} \tim{This takes constant time.} 
    \end{itemize}
    \item \textbf{Check whether $u_i(w)$ is $w$--reducible.}\label{step:backtrackcheck'} This involves two parts, we first check whether $t_{i-1}(w) = \bar{t}_{i} (w)$. If so, we use algorithm \ref{alg:in-egde-group-check} to check whether $u_i(w)\in \tau_{\bar{t}_i(w)}(G_{\bar{t}_i(w)})$. \tim{This takes at most $\abs{u_i(w)}$ time.}
        \begin{itemize}
            \item \textbf{ If $u_i(w)$ is not $w$--reducible:} Append $u_i(w)t_i(w)$ to $a$ and go to Step \ref{step:algrepeat'}. Namely, we have to append $t_i(w)$ to $\edge{\mf a}$, append $\# 1^{\norm{u_i(w)}_{w^0}}$ to $\leng{\mf a}$ and append $\# u_i(w)$ to $\word{\mf a}$. \tim{This takes at most $\norm{u_i(w)}_{w^0}$ time.} \gar{$new(a)$ is $(K, w^0)$-non-sprawling}
        \end{itemize}
    \item \textbf{Set the counter $c$ to $\norm{u_i(w)}_{w^0}$}\textit{and add $i_{start}$ to $I_{core}$.} We have to update $\mf c_{counter}$. \tim{ This takes at most $\norm{u_i(w)}_{w^0}$ time.} \label{step:set-counter'}
    \item\textbf{Set $w_{cur}$ to a word over $S_{t_i(w)}$ which satisfies $w_{cur} = _G t_a u_i(w) \bar{t}_a = \bar{t}_i(w) u_i(w) t_i(w)$ and remove $t_a$ from $a$.} We use Algorithm~\ref{alg:in-egde-group-check} on $(\bar{t}_i(w) u_i(w)t_i(w), t_i(w))$ which will output a decomposition $(w_{cur},t_{i}(w)) \equiv (k, \ell, h)$. We store the output on the bands $\mf c_e$. So while there are, in general, many possible choices for the word $w_{cur}$ \textbf{we choose $w_{cur}$ such that it is in edge normal form for the edge $t_i(w)$.} Next, we update $\mf c_{leng}$. Lastly, we delete the last edge from $\edge{\mf{a}}$. \tim{This takes at most $\norm{u_i(w)}_{w^0}$ time.}  \gar{ $\abs{w_{cur}}\leq K \norm{w_{cur}}_{w^0}$ and $c =2\norm{u_i(w)}_{w^0}- \norm{w_{cur}}_{w^0} $ and $w_{cur}$ is in edge normal form for the edge $new(t_a)$.} \label{step:curdef'}
    \item \textbf{Repeat the following}. \label{step:innerrep'}
    \begin{enumerate}[label = \roman*)]
        \item \textbf{Increase $i$.}\label{step:inner-increase'} \tim{This takes at most $\norm{u_{new(i)}(w)}_{w^0}$ time.}
        \item \textbf{Subtract $\norm{u_i(w)}_{w^0} + \norm{w_a}_{w^0}$ from $c$.} If $c$ is negative, we do not care about the precise value of it, so once $c$ becomes negative, we can stop subtracting. Define $c_{dif} =\min\{ \norm{u_i(w)}_{w^0} + \norm{w_a}_{w^0}, old(c)+1\}$, note that $c_{dif}$ is the exact amount we subtract from $c$. \tim{This step takes at most $c_{dif}$ time.} \label{step:update_counter2'} 
        \item \textbf{Check whether $c\geq 0$.} This takes constant time.\label{step:check-counter'}
        \begin{itemize}
            \item \textbf{If $c < 0$: Append $w_{cur}u_i(w)t_i$ to $a$. Go to step \ref{step:algrepeat'}.} Again, we have to update all bands in $\mf a_{\ast}$. Namely, we have to append $t_i(w)$ to $\edge{\mf a}$, append $1^{\norm{w_{cur}}_{w^0}+\norm{u_i(w)}_{w^0}}$ to $\leng{\mf a}$ and append $w_{cur}u_i(w)$ to $\word{\mf a}$. \tim{This takes at most $\norm{u_i(w)}_{w^0} + \norm{u_{i_{start}}(w)}_{w^0}$ time.} Lastly, we have to reset $c_{leng}$ and $c_{counter}$. \tim{This takes at most $\norm{u_{i_{start}}(w)}_{w^0}$ time.} Note that $\norm{w_a}_{w^0}$ increases by $\norm{u_i(w)}_{w^0} + \norm{w_{cur}}$ while $\abs{w_a}$ increases by $\abs{u_i(w)} +\abs{w_{cur}}\leq K(\norm{u_i(w)}_{w^0} + \norm{w_{cur}})$. \gar{$new(a)$ is $(K, w^0)$--non-sprawling.}
        \end{itemize}
        \item \textbf{Check whether the piece $w_aw_{cur}u_i(w)$ is $aw_{cur}w_i$--reducible.} Doing this efficiently is slightly involved and requires the following steps. \label{step:inner-reducibility-check'}
        \begin{enumerate}[label = \arabic*)]
            \item \textbf{Check whether  $t_a = \bar{t}_i(w)$. If not, go to step \ref{step:endofinnerstep'}.} If $t_a\neq \bar{t}_i(w)$, then $w_aw_{cur}u_i(w)$ is not $aw_{cur}w_i$--reducible. \tim{This takes constant time.}
            \item \textbf{Check whether $t_{i-1}(w) = \bar{t}_{i}(w)$,  $w_a\in \tau_{\bar{t}_i(w)}(G_{\bar{t}_i(w)})$ and $u_i(w)\in \tau_{\bar{t}_i(w)}(G_{\bar{t}_i})$.} We use algorithm \ref{alg:in-egde-group-check} to do this. \tim{This can be done in time at most $\abs{w_a} + \abs{u_i(w)}\preceq c_{dif}$.}
            \begin{itemize}
                \item \textbf{If so: Use algorithm \ref{alg:in-egde-group-check} to bring $w_a$ and $u_i(w)$ into edge normal form for edge $t_a = \bar{t}_i$ and add them to $w_{cur}$. Lastly, replace $w_{cur}$ with $t_aw_{cur}t_i(w)$. Go to Step \ref{step:innerrep'}} \textit{This works, since $old(w_{cur})$ is in edge normal form for the edge $\bar{t}_i(w) = t_{i-1}(w)$.} As described in Section \ref{sec:storing_in_edge_groups}, \tim{this can be done in time at most $\abs{w_a} + \abs{u_i(w)} \preceq c_{dif}$.} We also update $\mf c_{leng}$ since the derived length of $w_{cur}$ increases by $c_{dif}$. \tim{This takes at most $c_{dif}$ time.} \gar{$\abs{new(w_{cur})}\leq K \norm{new(w_{cur})}_{w^0}$, $c = 2\norm{u_{i_{start}}(w)}_{w^0} - \norm{w_{cur}}_{w^0}$ and $new(w_{cur})$ is in edge normal form for the edge $new(t_a)$.}
            \end{itemize}
            \item \textbf{Check whether $w_ao_{t_{i-1}(w)}^\ell hu_i(w)\in \tau_{t_a}(G_{t_a})$}. \textit{Here we define $\ell\in \Z$ and $h\in H_{t_a}$ such that $(w_{cur}, t_{i-1}(w)) \equiv (k, \ell, h)$. Note that 
            \begin{align*}
                 w_ao_{t_{i-1}(w)}^\ell hu_i(w)\in \tau_{t_a}(G_{t_a}) \quad \iff \quad  w_a w_{cur}u_i(w)\in \tau_{t_a}(G_{t_a})
            \end{align*}
            because $c_{t_{i-1}(w)}$ is in the center. We want to know whether the latter holds, but the former uses less time to check, which is why we check it instead.} We use algorithm \ref{alg:in-egde-group-check} to do this. \tim{This can be done in time at most $\abs{w_a}+\ell+ \abs{u_i(w)}$.} \label{step:check'}
            \begin{itemize}
                \item \textbf{If yes: Let $(k, \ell, h) \equiv (w_{cur} ,t_{i-1}(w))$ and $(w_ao_{t_{i-1}(w)}^\ell hu_i(w), t_{i-1}(w))\equiv (k', \ell', h')$. Replace $w_{cur}$ with the word in edge normal form for the edge $t_{i-1}(w)$ which is equivalent to $(k'+k, \ell', h')$. Lastly, replace $w_{cur}$ by $t_aw_{cur} t_i(w)$.} \textit{Since $c_{t_a} = c_{t_{i-1}(w)}$ is in the center, we have that $(k'+k, \ell', h') \equiv (w_a old(w_{cur}) u_i(w), t_a)$ and $new(w_{cur}) =  w_a old(w_{cur}) u_i(w)$.} \tim{This takes at most $\abs{w_a}+\abs{u_i(w)} +\ell$ time.} 
                We also have to update the derived $c_{leng}$ since the derived length of $w_{cur}$ increases by $c_{dif}$. \tim{This takes at most $ c_{dif}$ time.} \gar{$\abs{new(w_{cur})}\leq K \norm{new(w_{cur})}_{w^0}$, $w_{cur}$ is in edge normal form for the edge $t_{i}(w)$ and $c =2\norm{u_i(w)}_{w^0}- \norm{w_{cur}}_{w^0} $.}   
            \end{itemize}
             \item \textbf{Append $w_{cur}u_i(w)t_i(w)$ to $a$. Go to Step \ref{step:algrepeat'}.} \textit{We reach this step if and only if $w_aw_{cur}u_i(w)$ is not $aw_{cur}w_i$--reducible.} As in the case where $c<0$, we have to update all bands in $\mf a_{\ast}$ and reset the auxiliary bands. \tim{This takes at most $ \norm{u_i(w)}_{w^0} + \norm{u_{i_{start}}(w)}_{w^0}$ time.} \gar{$new(a)$ is $(K, w^0)$--non-sprawling.}\label{step:endofinnerstep'}
        \end{enumerate}
    \end{enumerate}

\end{enumerate}
\item \textbf{Define $w'$ as $a$.} \gar{$w'$ is $(K, w^0)$--non-sprawling.} \label{step:defofw''}
\end{enumerate}

In light of Lemma~\ref{lem:properties_of_vertex_groups}\eqref{prop:length_bound_for_edge_group}, all guarantees hold trivially. In light of Lemma~\ref{lemma:basic-properties-derivations}\eqref{prop:length-invariance}, to prove Lemma~\ref{lemma:compute_center_der_quickly}, it is enough to show that every iteration of \step can be done in time at most $\norm{u_{i_{start}}(w)}_{w^0} + \norm{u_{j-1}(w)}_{w^0}$, where $i_{start}$ is the instigator of that iteration and $j$ is the instigator of the next iteration.

Steps~\ref{step:1'} to \ref{step:set-counter'} can be done in time at most $\norm{u_i(w)}_{w^0}$. We claim that Step~\ref{step:innerrep''} takes at most $\norm{u_i(w)}_{w^0} + \norm{u_{j-1}(w)}_{w^0}$ time, which would conclude the proof.

 If the check in Step~\ref{step:check'} is positive, then by Lemma~\ref{lem:properties_of_vertex_groups}\eqref{prop:length_bound_for_edge_group}, $\ell\preceq \abs{w_a}+\abs{u_i(w)}\preceq c_{dif}$, so Step~\ref{step:check'} can be done in time at most $c_{dif}$. Thus, if a particular \sstep is not the last InnerStep, then it can be done in time at most $c_{dif}$. Since $c_{dif}$ is subtracted from our counter $c$ (whose initial value is $\norm{u_{i_{start}}(w)}_{w^0}$) and we stop repeating once $c$ becomes negative, performing all but the last of the \sstep iterations takes at most $\norm{u_{i_{start}}(w)}$ time.  Note that in Step~\ref{step:check'} $\ell\preceq \norm{w_{cur}}_{w^0}\leq 2\norm{u_{i_{start}}(w)}_{w^0}$, so if \sstep is the last step, we know it takes at most $\norm{u_{i_{start}}(w)}_{w^0} + \norm{u_{j-1}(w)}_{w^0}$, where the $\norm{u_{j-1}(w)}_{w^0}$ comes from Step~\ref{step:check-counter'}. So we indeed have the desired bound on the runtime of Step~\ref{step:innerrep'}, which concludes the proof of Lemma~\ref{lemma:compute_center_der_quickly}.

\subsection{Proof of Theorem \ref{thm:admissible}}

We now summarise how to prove Theorem \ref{thm:admissible} which states that the word problem for admissible graphs of groups can be solved in $O(n\log(n))$. We know that the word problem in the vertex groups can be solved in linear time by Proposition \ref{prop:central_word_problem}. Also, we know from Lemma \ref{lemma:compute_center_der_quickly} that ($(K,w^0)$-non-sprawling) middle derivations can be computed in linear time also. Hence, the word problem can be solved in $O(n\log(n))$ by Corollary \ref{cor:A12_implies_wp}, as required.\qed

\bibliographystyle{alpha}
\bibliography{biblio}

\end{document}